\documentclass[12pt]{article}

\usepackage[font=footnotesize,margin=1cm]{caption}

\usepackage[margin=1in]{geometry}

\usepackage{setspace}
\usepackage[utf8]{inputenc} 
\usepackage[T1]{fontenc}    
\usepackage{hyperref}       
\usepackage{url}            
\usepackage{booktabs}       
\usepackage{amsfonts}       
\usepackage{nicefrac}       
\usepackage{microtype}      

\usepackage{amsmath,amssymb,bbm,stmaryrd,mathrsfs,url,amsthm}
\newtheorem{theorem}{Theorem}
\newtheorem{lemma}{Lemma}

\usepackage{amstext}
\usepackage{array,xcolor}
\usepackage{color,soul}
\usepackage{hyperref}
\usepackage{blkarray}
\usepackage{amsmath}

\usepackage{amsmath}
\usepackage[all,cmtip]{xy}
\usepackage{booktabs,makecell}

\usepackage[T1]{fontenc}
\usepackage[utf8]{inputenc}
\usepackage{multirow,bigdelim}
\usepackage{makeidx}
\usepackage{overpic}
\makeindex

\usepackage[export]{adjustbox}
\usepackage{graphicx,floatrow}
\usepackage{subcaption}
\usepackage{natbib}
\DeclareMathOperator*{\argmin}{arg\,min}
\newcommand{\balpha}{{\boldsymbol{\alpha}}}
\newcommand{\bbeta}{{\boldsymbol{\beta}}}
\newcommand{\bxi}{{\boldsymbol{\xi}}}

\newcommand{\change}{}

\allowdisplaybreaks

\newcommand{\R}{\mathbb{R}}

\newcommand{\E}{\operatorname{E}}

\newcommand{\vct}[1]{\boldsymbol{#1}}
\newcommand{\mtx}[1]{\boldsymbol{#1}}

\newcommand{\<}{\langle}
\renewcommand{\>}{\rangle}

\newcommand{\set}[1]{\mathcal{#1}}

\DeclareMathOperator*{\minimize}{\text{minimize}}

\newcommand{\vb}{\vct{b}}
\newcommand{\vc}{\vct{c}}

\newcommand{\vg}{\vct{g}}
\newcommand{\vh}{\vct{h}}

\newcommand{\vm}{\vct{m}}

\newcommand{\vs}{\vct{s}}
\newcommand{\vt}{\vct{t}}
\newcommand{\vu}{\vct{u}}
\newcommand{\vv}{\vct{v}}
\newcommand{\vw}{\vct{w}}
\newcommand{\vx}{\vct{x}}
\newcommand{\vy}{\vct{y}}
\newcommand{\vz}{\vct{z}}
\newcommand{\valpha}{\vct{\alpha}}
\newcommand{\veta}{\vct{\eta}}

\newcommand{\vbeta}{\vct{\beta}}
\newcommand{\vzeta}{\vct{\zeta}}

\newcommand{\vxi}{\vct{\xi}}

\newcommand{\vzero}{\vct{0}}

\newcommand{\mA}{\mtx{A}}
\newcommand{\mB}{\mtx{B}}
\newcommand{\mC}{\mtx{C}}
\newcommand{\mD}{\mtx{D}}
\newcommand{\mE}{\mtx{E}}
\newcommand{\mF}{\mtx{F}}

\newcommand{\mI}{\mtx{I}}

\newcommand{\mP}{\mtx{P}}
\newcommand{\mQ}{\mtx{Q}}

\newcommand{\setD}{\set{D}}

\newcommand{\setN}{\set{N}}

\newcommand{\setR}{\set{R}}

\newcommand{\yl}{y_\ell}
\renewcommand{\sl}{s_\ell}
\newcommand{\wl}{w_\ell}
\newcommand{\xl}{x_\ell}

\newcommand{\bl}{\vb_\ell}

\newcommand{\cl}{\vc_\ell}

\newcommand{\dm}{\delta \vm}
\renewcommand{\dh}{\delta \vh}

\DeclareMathOperator{\sign}{sign}

\def\l{\ell}

\newcommand{\ho}{\vh^\natural}

\newcommand{\mo}{\vm^\natural}

\newcommand{\wo}{\vw^\natural}
\newcommand{\xo}{\vx^{\natural}}

\newcommand{\blt}{\bl^\intercal}
\newcommand{\clt}{\cl^\intercal}

\newcommand{\tildeh}{\tilde{\vh}}
\newcommand{\tildem}{\tilde{\vm}}

\newcommand{\hath}{\hat{\vh}}
\newcommand{\hatm}{\hat{\vm}}

\newcommand{\PP}{\mathbb{P}}

\title{Bilinear Compressed Sensing under known Signs via Convex Programming}

\author{
  Alireza Aghasi\thanks{
  aaghasi@gsu.edu, J. Mack Robinson College of Business, GSU}, Ali Ahmed\thanks{ali.ahmed@itu.edu.pk, Department of Electrical Engineering, ITU, Lahore}, Paul Hand\thanks{p.hand@northeastern.edu, Department of Mathematics and Khoury College of Computing Science, Northeastern University}\ \ and Babhru Joshi\thanks{babhru.joshi@rice.edu, Department of Computational and Applied Mathematics, Rice University}
}
\begin{document}

\maketitle
\begin{abstract}
  We consider the bilinear inverse problem of recovering two vectors, $\vx\in\R^L$ and $\vw\in\R^L$, from their entrywise product. We consider the case where $\vx$ and $\vw$ have known signs and are sparse with respect to known dictionaries of size $K$ and $N$, respectively.  Here,  $K$ and $N$ may be larger than, smaller than, or equal to $L$. We introduce $\ell_1$-BranchHull, which is a convex program posed in the natural parameter space and does not require an approximate solution or initialization in order to be stated or solved. Under the assumptions that $\vx$ and $\vw$ satisfy a comparable-effective-sparsity condition and are $S_1$- and $S_2$-sparse with respect to a random dictionary, we present a recovery guarantee in a noisy case. We show that $\ell_1$-BranchHull is robust to small dense noise with high probability if the number of measurements satisfy $L\geq\Omega\left((S_1+S_2)\log^{2}(K+N)\right)$. Numerical experiments show that the scaling constant in the theorem is not too large. We also introduce variants of $\ell_1$-BranchHull for the purposes of tolerating noise and outliers, and for the purpose of recovering piecewise constant signals.  We provide an ADMM implementation of these variants and show they can extract piecewise constant behavior from real images.
 \end{abstract}

\section{Introduction}

We study the problem of recovering two unknown vectorss $\vx$ and $\vw$ in $\R^{L}$ from observations 
\begin{equation}\label{eq:measurements}
\vy =\vw\odot\vx\odot (\boldsymbol{1} +\vxi)	,
\end{equation}
where $\odot$ denotes entrywise product and $\vxi \in \mathbb{R}^L$ is noise. Let $\mB \in \R^{L\times K}$ and $\mC \in \R^{L\times N}$  such that $\vw = \mB\vh$ and $\vx =\mC\vm$ with $\|\vh\|_0 \leq S_1$ and $\|\vm\|_0\leq S_2$. The bilinear inverse problem (BIP) we consider is to find $\vw$ and $\vx$, up to the inherent scaling ambiguity, from $\vy$, $\mB$, $\mC$ and $\sign{(\vw)}$. 

BIPs, in general, have many  applications in signal processing  and machine learning and include fundamental practical problems like phase retrieval \citep{fienup1982phase, candes2012solving, candes2013phaselift}, blind deconvolution \citep{ahmed2012blind, stockham1975blind, kundur1996blind,aghasi2016sweep}, non-negative matrix factorization \citep{hoyer2004non,lee2001algorithms}, self-calibration \citep{ling2015self}, blind source separation \citep{Gardy05source}, dictionary learning \citep{tosic2011dictionary}, etc. These problems are in general challenging and suffer from identifiability issues that make the solution set non-unique and non-convex. A common identifiability issue, also shared by the BIP in \eqref{eq:measurements}, is the scaling ambiguity. In particular, if $(\wo,\xo)$ solves a BIP, then so does $(c\wo, c^{-1}\xo)$ for any nonzero $c \in \R$. In this paper, we resolve this scaling ambiguity by finding the point in the solution set closest to the origin with respect to the $\ell_1$ norm.

Another identifiability issue of the BIP in \eqref{eq:measurements} is if $(\wo,\xo)$ solves \eqref{eq:measurements}, then so does $(\bf{1}, \wo \odot \xo)$, where $\bf{1}$ is the vector of ones. In prior works like \cite{ahmed2012blind}, this identifiability issue is resolved by assuming the unknown signals live in known subspaces. In contrast we resolve the identifiability issue by assuming the signals are sparse with respect to known bases or dictionaries. Natural choices for such bases include the standard basis, the Discrete Cosine Transform (DCT) basis, and a wavelet basis. In addition to sparsity of the unknown vectors, $\wo$ and $\xo$, with respect to known dictionaries, we also require sign information of $\wo$ and $\xo$. Knowing sign information is justified in, for example, imaging applications where $\wo$ and $\xo$ correspond to images, which contain non-negative pixel values.

In this paper, we study a sparse bilinear inverse problem, where the unknown vectors are sparse with respect to known dictionaries. Similar sparse BIPs have been extensively studied in the literature and are known to be challenging. In particular, the best known algorithm for sparse BIPs that can provably recovery $\ho$ and $\mo$ require measurements that scale quadratically, up to log factors, with respect to the sparsity levels, $S_1$ and $S_2$.  Recent work on sparse rank-1 matrix recovery problem in \cite{Bresler2016sparse}, which is motivated by considering the lifted version of the sparse blind deconvolution problem, provides an exact recovery guarantee of the sparse vectors $\ho$ and $\mo$ that satisfy a `peakiness' condition, i.e. $\min\{\|\ho\|_\infty,\|\mo\|_\infty\}\geq c$ for some absolute constant $c \in \R$, using a non-convex approach. This result holds with high probability for random measurements if the number of measurements satisfy $L \geq\Omega(S_1+S_2)$, up to a log factor. For general vectors that are not constrained to a class of sparse vectors like those satisfying the peakiness condition, the same work shows exact recovery is possible if the number of measurements satisfy $L\geq \Omega(S_1S_2)$, up to a log factor. 

The main contribution of the present paper is to introduce a convex program in the natural parameter space for the sparse BIP described in \eqref{eq:measurements} and show that it can stably recover sparse vectors, up to the global scaling ambiguity, provided they satisfy a comparable-effective-sparsity condition. Precisely, we say the sparse vectors $\ho$ and $\mo$ are $\rho$-comparable-effective-sparse, for some $\rho\geq 1$, if there exists an $\alpha\in\R$ that satisfies $\max\left(\alpha,1/\alpha\right) \leq \rho$ and 
\begin{equation}\label{cond:eff sparsity}
	\frac{\|\ho\|_1}{\|\ho\|_2} =\alpha \frac{\|\mo\|_1}{\|\mo\|_2}.
\end{equation}
Note that the ratio of the $\ell_1$ to $\ell_2$ norm of a vector is an approximate notion of the square root of the sparsity of a vector. Additionally, we assume the noise in \eqref{eq:measurements} does not change the sign of the measurements. Specifically, we consider noise $\vxi\in \mathbb{R}^L$ such that 
\begin{equation}\label{noise-condition}
	\xi_\ell \geq -1 \text{ for all } \ell = 1,\dots,L.
\end{equation}

\begin{figure}[t]
\floatsetup[subfigure]{captionskip=-20pt} 
\ffigbox[]{\hspace{0pt}
    \begin{subfloatrow}[2]
      \ffigbox[]{
      \begin{overpic}[width=0.5\textwidth,height=0.4\textwidth,tics=1]{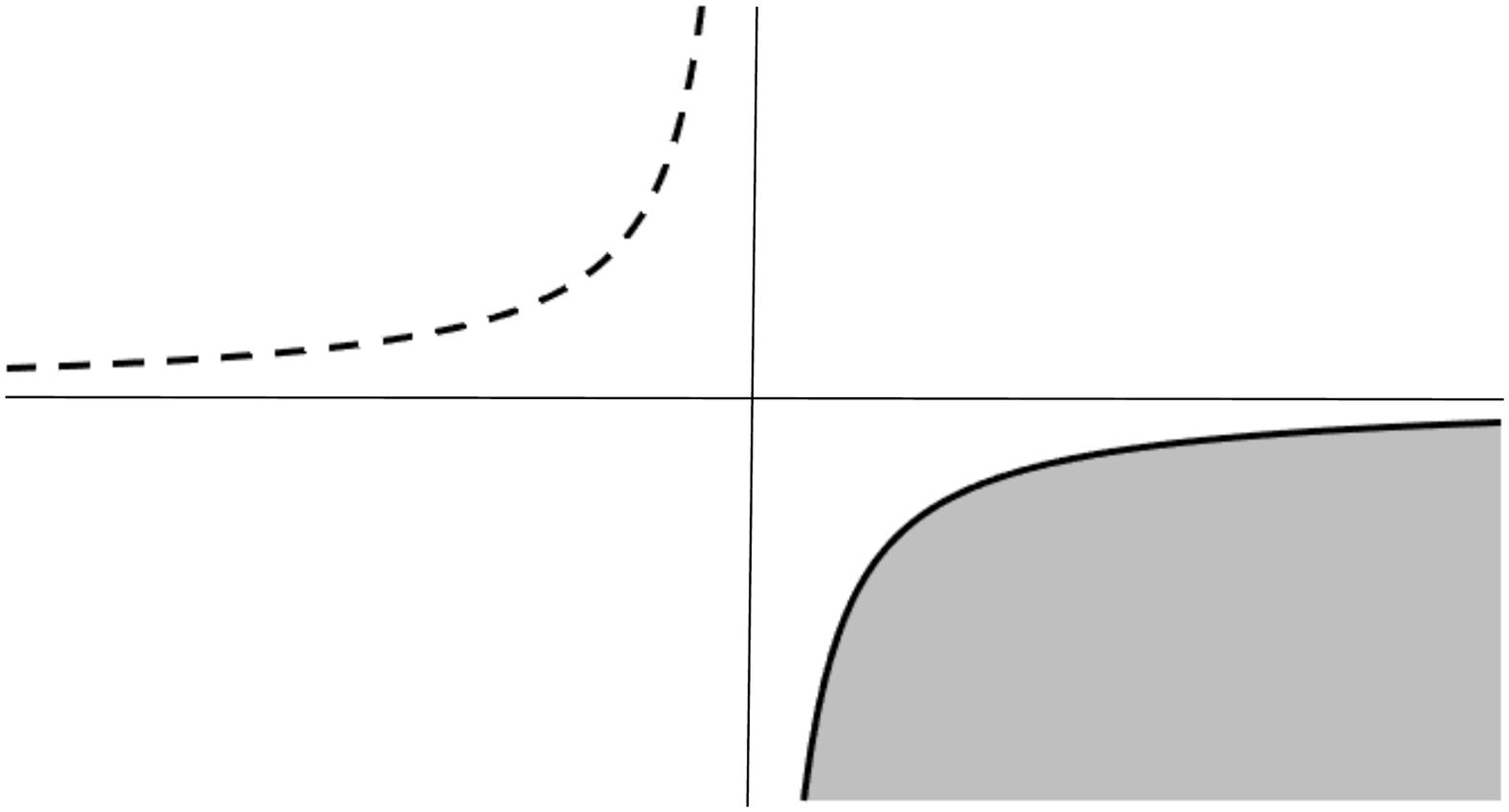}
      	\put(49,65){$w_\ell$}
     	\put(51,41){$0$}
		\put(93,40){$x_\ell$}
		\put(49,28){\rotatebox{35}{$x_\ell w_\ell =y_\ell $}}
		\put(58,24.5){\rotatebox{0}{\scalebox{.93}{Convex Hull}}}
	  \end{overpic}
        }{\caption{Convex relaxation}\label{fig:hyperbola}}
      \hspace{\fill}\ffigbox[\FBwidth]{\raisebox{.8cm}{
        \begin{overpic}[scale = 0.4, trim= 4cm 11cm 6cm 4cm, clip = true]{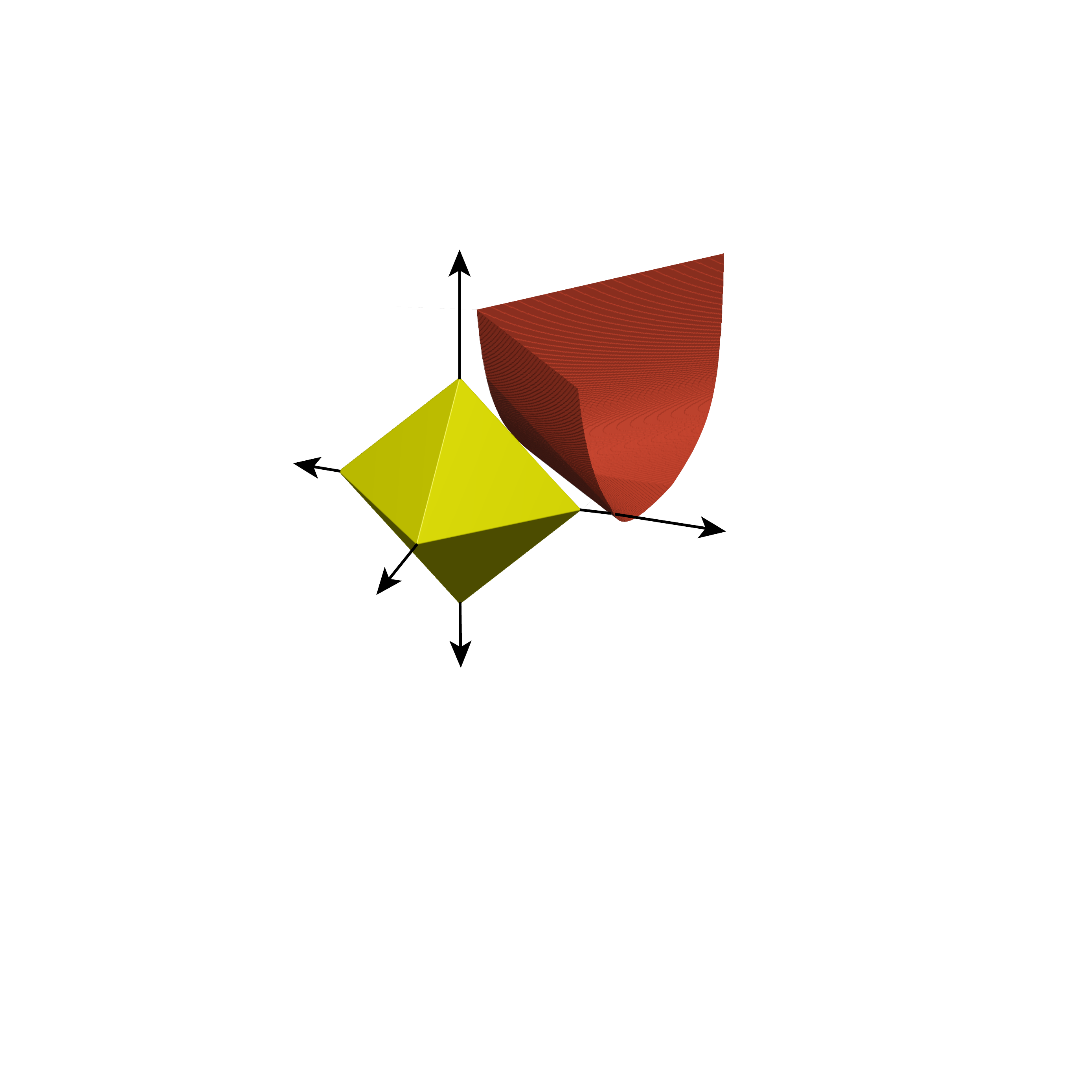}
		\put(28,5){$h_1$}
		\put(83,18){$h_2$}
		\put(41,62){$m_1$}
	\end{overpic}}}{\caption{Geometry of $\ell_1$-BranchHull}\label{fig:geometry}}
    \end{subfloatrow}
}{\caption{Panel (a) shows the convex hull of the relevant branch of a hyperbola given a measurement $\yl$ and the sign information $\sign(\wl)$. Panel (b) shows the interaction between the $\ell_1$-ball in the objective of $\eqref{eq:BH}$ with its feasibility set.  The feasibility set is `pointy' along a hyperbola corresponding the scaling ambiguity, which allows for signal recovery where the $\ell_1$ ball touches this hyperbola.}
}
\end{figure}

Under the assumptions that the sparse vectorss satisfy \eqref{cond:eff sparsity} and the noise is small as in \eqref{noise-condition}, we show that the convex program stably recovers the unknown vectors, up to the global scaling ambiguity, with high probability provided $\mB$ and $\mC$ are random and the number measurements satisfy $L \geq\Omega\left(S_1+S_2\right)$, up to log factors. Similar to the result in \cite{Bresler2016sparse}, this results has optimal sample complexity, up to log factors, for a class of sparse signals, namely those with comparable sparsity levels.

\subsection{Convex program and main results}

We introduce a convex program written in the natural parameter space for the bilinear inverse problem described in \eqref{eq:measurements}. Let $(\ho,\mo) \in \R^{K}\times \R^{N}$ with $\|\ho\|_0 \leq S_1$ and $\|\mo\|_0\leq S_2$.  Let $w_\ell = \blt \ho$, $x_\ell = \clt \mo$ and $y_\ell = \blt\ho\cdot \clt\mo\cdot(1+\xi_\l)$, where $\blt$ and $\clt$ are the $\ell$th row of $\mB$ and $\mC$, respectively, and $\xi_\l$ is the $\ell$th entry of $\vxi$. Also, let $\vs = \sign(\vy)$ and $\vt = \sign(\vw)=\sign(\mB\ho)$. The convex program we consider to recover $(\ho,\mo)$ is the $\ell_1$-BranchHull program
\begin{align}
\text{$\ell_1$-BH}:\ &\underset{\vh\in \R^K, \vm \in \R^N}{\minimize}~\|\vh\|_1+\|\vm\|_1\label{eq:BH}\\&\text{subject to}~~ \sl(\blt \vh \clt\vm) \geq |\yl|\notag\\
&\qquad\qquad\qquad t_\ell\blt\vh \geq 0, \quad \ell = 1,2,\ldots,L.\notag 
\end{align}
The motivation for the feasible set in program \eqref{eq:BH} follows from the observation that each measurement $y_\ell = w_\ell \cdot x_\ell$ defines a hyperbola in $\R^2$. As shown in Figure \ref{fig:hyperbola}, the sign information $t_\ell = \sign(w_\ell)$ restricts $(w_\ell,x_\ell)$ to one of the branches of the hyperbola. The feasible set in \eqref{eq:BH} corresponds to the convex hull of {particular branches of the hyperbola} for each $y_\ell$. This also implies that the feasible set is convex as it is the intersection of $L$ convex sets.

The objective function in \eqref{eq:BH} is an $\ell_1$ minimization over $(\vh,\vm)$ that finds a point $(\hath, \hatm)$ with $\|\hath\|_1 = \|\hatm\|_1$. Geometrically, this happens as the solution lies at the intersection of the $\ell_1$-ball and the hyperbolic curve (constraint) as shown in Figure \ref{fig:geometry}. So, the minimizer of \eqref{eq:BH} in the noiseless case, under successful recovery, is $(\hath,\hatm):=\left(\ho \sqrt{\frac{\|\mo\|_1}{\|\ho\|_1}},\mo \sqrt{\frac{\|\ho\|_1}{\|\mo\|_1}}\right)$. 

We now present our main result which states that the $\ell_1$-BranchHull program \eqref{eq:BH} stably recovers $\vw$ and $\vx$, up to the global scaling ambiguity, in the presence of small dense noise. We show that if $\vw$ and $\vx$ live in random subspaces with $\ho$ and $\mo$ containing at most $S_1$ and $S_2$ non zero entries, $\ho$ and $\mo$ satisfy comparable-effective-sparsity condition \eqref{cond:eff sparsity}, the noise $\vxi$ satisfy \eqref{noise-condition}, and there are at least $\Omega((S_1+S_2)\log^2(K+N))$ number of measurements, then the minimizer of the $\ell_1$-BranchHull program is close to the bilinear ambiguity set $\{(c\hath,c^{-1}\hatm)|c>0\}$ with high probability. Moreover, in the case noiseless case, the minimizer of the $\ell_1$-BranchHull is the point on this bilinear ambiguity set with equal $\l_1$ norm with high probability, i.e. the minimizer is $(\hath,\hatm)$ with high probability.

\begin{theorem}[Noisy recovery]\label{thm:Noisy_Main}
	Fix $\rho \geq 1$. Fix $(\ho,\mo) \in \mathbb{R}^{K+N}$ that are $\rho$-comparable-effective-sparse as defined in \eqref{cond:eff sparsity}  with $\ho\neq \bf{0}$, $\|\ho\|_0\leq S_1$ and $\mo\neq \bf{0}$, $\|\mo\|_0\leq S_2$. Let $\mB\in \mathbb{R}^{L\times K}$ and $\mC \in \mathbb{R}^{L\times N}$ have i.i.d. $\mathcal{N}(0,1)$ entries. Let $\vy \in \mathbb{R}^{L}$ contain measurements that satisfy \eqref{eq:measurements} with noise $\vxi \in \mathbb{R}^{L}$ satisfying \eqref{noise-condition}. If $L \geq 	C_\rho\left(\sqrt{S_1+S_2}\log(K+N)+t\right)^2$ for any $t>0$ then the $\ell_1$-BranchHull program \eqref{eq:BH} recovers $(\tildeh,\tildem)$ that satisfies 
	\begin{align*}
		\min_{c>0}\frac{\left\|(\tildeh,\tildem)-(c\hath,c^{-1}\hatm)\right\|_2}{\|(c\hath,c^{-1}\hatm)\|_{2}}\leq 37\sqrt{\|\vxi\|_\infty}
	\end{align*}
with probability at least $1-\mathrm{e}^{-c_t L}$. Here, $C_\rho$ and $c_t$ are constants that depend quadratically on $\rho$ and $t$, respectively. Furthermore, $(\tildeh,\tildem) = (\hath,\hatm)$ if $\vxi = \vzero$. 
\end{theorem} 

Theorem \ref{thm:Noisy_Main} shows that exact recovery, up to the global scaling ambiguity, of sparse vectors that satisfy the comparable-effective-sparsity condition is possible if the number of measurements satisfy $L \geq \Omega\left((S_1+S_2)\log^2(K+N)\right)$. This result is optimal, up to the log factors. Numerical simulation on synthetic data verify Theorem \ref{thm:Noisy_Main} in the noiseless case and show that the constant in the sample complexity is not too large. We also present the results of numerical simulation on two real images which shows that a total variation reformulation of the convex program \eqref{eq:BH} can successfully recover the piecewise constant part of an otherwise distorted image.

\subsection{Prior art for bilinear inverse problems} 

Recent approaches to solving bilinear inverse problems like blind deconvolution and phase retrieval include lifting the problems into a low rank matrix recovery task or to formulate a convex or non-convex optimization programs in the natural parameter space. Lifting transforms the problem of recovering $\vh \in \R^K$ and $\vm \in \R^N$ from bilinear measurements to the problem of recovering a low rank matrix $\vh\vm^\intercal$ from linear measurements. The low rank matrix can then be recovered using a semidefinite program. The result in \cite{ahmed2012blind} for blind deconvolution showed that if $\vh$ and $\vm$ are coefficients of the target signals with respect to Fourier and Gaussian bases, respectively, then the lifting method successfully recovers the low rank matrix. The recovery occurs with high probability under near optimal sample complexity. Unfortunately, solving the semidefinite program is prohibitively computationally expensive because they operate in high-dimension space. Also, it is not clear whether or not it is possible to enforce additional structure like sparsity of $\vh$ and $\vm$ in the lifted formulation in a way that allows optimal sample complexity \citep{li2013sparse, oymak2015simultaneously}. 

In comparison to the lifting approach for blind deconvolution and phase retrieval, methods that formulate an algorithm in the natural parameter space, such as alternating minimization and gradient descent based method, are computationally efficient and also enjoy rigorous recovery guarantees under optimal or near optimal sample complexity \citep{li2016rapid, candes2014phase, netrapalli2013phase, sun2016geometric,Bresler2016sparse}. The work in \cite{Bresler2016sparse} for sparse blind deconvolution is based on alternating minimization. In the paper, the authors use an alternating minimization that successively approximate the sparse vectors while enforcing the low rank property of the lifted matrix. However, because these methods are non-convex, convergence to the global optimal requires a good initialization \citep{tu2015low, CC15, li2016rapid}. 

Other approaches that operate in the natural parameter space include PhaseMax \citep{bahmani2016phase,  goldstein2016phasemax} and BranchHull \citep{aghasi2017branchHull}. PhaseMax is a linear program which has been proven to find the target signal in phase retrieval under optimal sample complexity if a good anchor vector is available. As with alternating minimization and gradient descent based approach, PhaseMax requires a good initializer to even be stated. In PhaseMax the initialization is part of the objective function but in alternating minimization the initialization is part of the algorithmic implementation. BranchHull is a convex program which solves the BIP described in \eqref{eq:measurements} in the dense signal case under optimal sample complexity. Like the $\ell_1$-BranchHull presented in this paper, BranchHull does not require an initialization but does require the sign information of the signals.

The $\ell_1$-BranchHull program \eqref{eq:BH} combines strengths of both the lifting method and the gradient descent based method. Specifically, the $\ell_1$-BranchHull program is a convex program that operates in the natural parameter space, without a need for an initialization. These strengths are achieved at the cost of the sign information of the target signals $\vw$ and $\vx$, which can be justified in imaging applications where the goal might be to recover pixel values of a target image, which are non-negative.

\subsection{Discussion and extensions}\label{total variation}
The $\l_1$-BranchHull formulation is inspired by the BrachHull formation  introduced in \cite{aghasi2017branchHull}, which is a novel convex relaxation for the bilinear recovery of the entrywise product of vectors with known signs, and share many of its advantages and drawbacks. Like in BranchHull, $\l_1$-BranchHull finds a point in the convex feasibility set that is closest to the origin. The important difference between these formulations is that BranchHull finds the point with the least $\l_2$ norm while $\l_1$-BranchHull finds the point with the least $\l_1$-norm. Another difference is that BranchHull enjoys recovery guarantee when those vectors belong to random real subspaces while $\l_1$-BranchHull enjoys recovery guarantee under a much weaker condition of those vectors admitting a sparse representation with respect to known dictionaries. 

Similar to BranchHull, $\l_1$-BranchHull is a flexible and can altered to tolerate large outliers. In this paper, we show that the $\l_1$-BranchHull formulation is stable to small dense noise. However, as stated formulation \eqref{eq:BH} is not robust to outliers. This is because the formulations is particularly susceptible to noise that changes the sign of even a single measurement. For the bilinear inverse problem as described in \eqref{eq:measurements} with small dense noise and arbitrary outliers, we propose the following robust $\ell_1$-BranchHull program by adding a slack variable.
\begin{align}
	\text{$\l_1$-RBH:}&\minimize_{\vh \in \R^K, \vm \in \R^N,\vxi \in \R^L} \|\vh\|_1+\|\vm\|_1+\lambda\|\vxi\|_{1}\label{eq:RBH}\\
	&\text{subject to }s_\ell(\clt \vm + \xi_{\l})\blt \vh \geq |y_{\l}|,\notag\\ 
	 &t_\ell \blt \vh \geq 0, \quad \l = 1,\dots, L \notag.
\end{align}	
For measurements $y_\ell$ with incorrect sign, the corresponding slack variables $\xi_\ell$ shifts the feasible set so that the target signal is feasible. In the outlier case, the $\l_1$ penalty promotes sparsity of slack variable $\vxi$. We implement a slight variation of the above program to remove distortions from an otherwise piecewise constant signal. In the case where $\vw = \mB\ho$ is a piecewise constant signal, $\vx = \mC\mo$ is a distortion signal and $\vy = \vw \odot\vx$ is the distorted signal, the total variation version \eqref{eq:TVBH} of the robust BranchHull program \eqref{eq:RBH}, under successful recovery, produces the piecewise constant signal $\mB\ho$, up to a scaling.
\begin{align}\label{eq:TVBH}
\text{TV BH}: &\minimize_{\vh \in \R^K, \vm \in \R^N,\vxi \in \R^L}\hspace{-15pt}~\mbox{TV}\left(\mB\vh\right)+\|\vm\|_1+\lambda \|\boldsymbol{\xi}\|_1\\
 &\text{subject to}~~ s_\ell(\xi_\ell+\vc_\ell^\top\vm)\vb_\ell ^\top \vh \geq |\yl|\notag\\
&t_\ell\vb_\ell^\top \vh \geq 0, \quad \ell = 1,2,\ldots,L.\notag 
\end{align}
In \eqref{eq:TVBH}, TV$(\cdot)$ is a total variation operator and is the $\ell_1$ norm of the vector containing pairwise difference of neighboring elements of the target signal $\mB\vh$. We implement \eqref{eq:TVBH} to remove distortions from images in Section \ref{experiments} and leave detailed theoretical analysis of robust $\l_1$ BranchHull \eqref{eq:RBH} and its variant \eqref{eq:TVBH} to future work. It would also be interesting to develop convex relaxations in the natural parameter space that do not require sign information and to extend the analysis to the case when the phases of complex vectors are known and to the case of deterministic dictionaries instead of random dictionaries. All of these directions are left for future research.

%

\subsection{Organization of the paper}
The remainder of the paper is organized as follows. In Section \ref{notation}, we present notations used throughout the paper. In Section \ref{algorithm}, we present an Alternating Direction of Multipliers implementation of robust $\ell_1$-BranchHull program \eqref{eq:RBH}. In Section \ref{experiments}, we observe the performance of $\l_1$-BranchHull on synthetic random data and natural images. In Section \ref{proof}, we present the proof of Theorem \ref{thm:Noisy_Main}.

\subsection{Notation}\label{notation}
Vectors and matrices are written with boldface, while scalars and entries of vectors are written in plain font.  For example, $c_\ell$ is the $\ell$the entry of the vector $\vc$.  We write $\boldsymbol{1}$ as the vector of all ones with dimensionality  appropriate for the context. We write $\mI_N$ as the $N\times N$ identity matrix. For any $x \in \R$, let $(x)_- \in \mathbb{Z}$ such that $x-1<(x)_-\leq x$. For any matrix $\mA$, let $\|\mA\|_F$ be the Frobenius norm of $\mA$. For any vector $\vx$, let $\|\vx\|_0$ be the number of non-zero entries in $\vx$. For $\vx \in \R^K$ and $\vy \in \R^N$, $(\vx,\vy)$ is the corresponding vector in $\R^K \times \R^N$, and $\<(\vx_1,\vy_1),(\vx_2,\vy_2)\> = \<\vx_1,\vx_2\> + \<\vy_1,\vy_2\>$. 

\section{Algorithm}\label{algorithm}
In this section, we present an Alternating Direction Method of Multipliers (ADMM) implementation of an extension of the robust $\ell_1$-BranchHull program \eqref{eq:RBH}. The ADMM implementation of the $\ell_1$-BranchHull program \eqref{eq:BH} is similar to the ADMM implementation of \eqref{eq:Rl1BH} and we leave it to the readers. The extension of the robust $\ell_1$-BranchHull program we consider is 
\begin{align}\label{eq:Rl1BH}
&\underset{\vh \in \R^K, \vm \in \R^N,\vxi \in \R^L}{\minimize}~\|\mP\vh\|_1+\|\vm\|_1+\lambda \|\boldsymbol{\xi}\|_1\\ &\text{subject to}~~ s_\ell(\xi_\ell+\vc_\ell^\top\vm)\vb_\ell ^\top \vh \geq |\yl|\notag\\
&\qquad\qquad t_\ell\vb_\ell^\top \vh \geq 0, \quad \ell = 1,2,\ldots,L,\notag 
\end{align}
where $\mP \in \R^{J \times K}$. The above extension reduces to the robust $\ell_1$-BranchHull program if $\mP = \mI_K$. Recalling that $\vw = \mB\vh$ and $\vx = \mC\vm$, we form the block diagonal matrices $\mE = \mbox{diag}(\mC,\mB, \lambda^{-1}\mI_L)$ and $\mQ = \mbox{diag}(\mI_N,\mP,\mI_L)$ and define the vectors
\begin{align*}
&\vu = \begin{pmatrix} \vx\\ \vw\\ \boldsymbol{\xi}\end{pmatrix},~~ \vv = \begin{pmatrix} \vm\\ \vh\\ \lambda\boldsymbol{\xi}\end{pmatrix}.
\end{align*}
Using this notation, our convex program can be compactly written as
\begin{align}\underset{\vv\in \R^{N+K+L},\vu\in \R^{3L}}{\minimize}~~\|\mQ\vv\|_1~~\text{subject to}~~ \vu = \mE\vv ,~~ \vu\in \mathcal{C}. \notag
\end{align}
Here $\mathcal{C} = \{(\vx,\vw,\vxi)\in\R^{3L}|\ s_\ell(\xi_\ell+x_\ell)w_\ell\geq|y_\ell|, \ t_\ell w_\ell\geq 0,\ \l = 1,\dots, L\}$ is the convex feasible set of $\eqref{eq:Rl1BH}$. Introducing a new variable $\vz$ the resulting convex program can be written as
\[\underset{\vv,\vu,\vz}{\minimize}~~\|\vz\|_1~~\text{subject to}~~ \vu = \mE\vv, ~~ \mQ\vv = \vz ,~~ \vu\in \mathcal{C}. \notag
\]
We may now form the scaled ADMM through the following steps
\begin{align}\label{admm1}
\vu_{k+1} &= \argmin_{\vu\in\mathcal{C}}~~ \frac{\rho}{2}\left\|\vu + \balpha_k   - \mE\vv_k \right\|^2\\\label{admm2}
\vz_{k+1} & = \argmin_{\vz}~~ \|\vz\|_1 + \frac{\rho}{2}\left\|\vz + \bbeta_k  - \mQ\vv_k \right\|^2\\
\vv_{k+1} &= \argmin_{\vv}~~ \frac{\rho}{2}\left\|\balpha_k + \vu_{k+1} - \mE\vv \right\|^2 + \frac{\rho}{2}\left\|\bbeta_k + \vz_{k+1} - \mQ\vv \right\|^2,\label{admm3}
\end{align}
which are followed by the dual updates
\begin{align*}\balpha_{k+1} &= \balpha_k + \vu_{k+1} - \mE\vv_{k+1},\\ \bbeta_{k+1} &= \bbeta_k + \vv_{k+1} - \mQ\vv_{k+1}.
\end{align*}
We would like to note that the first three steps of the proposed ADMM scheme can be presented in closed form. The update in \eqref{admm1} is the following projection
\[\vu_{k+1} = \mbox{proj}_{\mathcal{C}}\left(  \mE\vv_k - \balpha_k \right),
\]
where $\mbox{proj}_{\mathcal{C}}(\vv)$ is the projection of $\vv$ onto $\mathcal{C}$. Details of computing the projection onto $\mathcal{C}$ are presented in Appendix \ref{sec:proj}. The update in \eqref{admm2} can be written in terms of the soft-thresholding operator
\begin{align*}
\vz_{k+1} = S_{1/\rho} \left( \mQ\vv_k - \bbeta_k \right),
\end{align*}
where for  $c>0$, 
\begin{align*}
\left(S_{c}(\vz)\right)_i = \left\{ \begin{array}{cl} z_i-c & z_i>c\\ 0 & |z_i|\leq c\\ z_i+c & z_i<-c\end{array}\right., 
\end{align*} 
and $\left(S_{c}(\vz)\right)_i$ is the $i$th entry of $S_{c}(\vz)$. Finally, the update in \eqref{admm3} takes the following form
\begin{align*}
\vv_{k+1} = &\left(\mE^\top\mE + \mQ^\intercal\mQ \right)^{-1}\bigg(\mE^\top\left( \balpha_k + \vu_{k+1}\right) +\mQ^\top(\bbeta_k + \vz_{k+1})\bigg).
\end{align*}
In our implementation of the ADMM scheme, we initialize the algorithm with the $\vv_0 = \bf{0}$, $\valpha_0 = \bf{0}$, $\vbeta_0 = \bf{0}$.

\subsection{Evaluation of the Projection Operator}\label{sec:proj}

Given a point $(\vx' ,\vw',\bxi')\in\mathbb{R}^{3L}$, in this section we focus on deriving a closed-form expression for $\mbox{proj}_{\mathcal{C}}\left( (\vx' ,\vw',\bxi')\right)$, where 
\[C = \left\{(\vx,\vw,\vxi)\in\R^{3L}|\ s_\ell(\xi_\ell+x_\ell)w_\ell\geq |y_\ell|,\ t_\ell w_\ell\geq 0,\ \l = 1,\dots, L\right\}\]
 is the convex feasible set of $\eqref{eq:Rl1BH}$. It is straightforward to see that the resulting projection program decouples into $L$ convex programs in $\mathbb{R}^3$ as
\begin{equation}\label{projeq}\argmin_{x\in\R,w\in\R,\xi\in\R}~~ \frac{1}{2}\left\|\begin{pmatrix}x\\ w\\ \xi \end{pmatrix} -  \begin{pmatrix} x_\ell'\\ w_\ell'\\ \xi_\ell'  \end{pmatrix}\right\|_2^2 ~~s.t.~~ |y_\ell| - s_\ell xw - s_\ell \xi  w  \leq 0,\quad  -t_\ell w \leq 0.
\end{equation}
Throughout this derivation we assume that $|y_\ell|>0$ (derivation of the projection for the case $y_\ell$ is easy) and as a result of which the second constraint $-t_\ell w \leq 0$ is never active (because then $w=0$ and the first constraint requires that $|y_\ell|\leq 0$). We also consistently use the fact that $t_\ell$ and $s_\ell$ are signs and nonzero. 

Forming the Lagrangian as
\[\mathcal{L}(x,w,\xi,\mu_1,\mu_2) = \frac{1}{2}\left\|\begin{pmatrix}x\\ w\\ \xi \end{pmatrix} -  \begin{pmatrix} x_\ell'\\ w_\ell'\\ \xi_\ell'  \end{pmatrix}\right\|_2^2 + \mu_1\left( |y_\ell| - s_\ell xw - s_\ell \xi w \right ) - \mu_2\left( t_\ell w \right),
\]
along with the primal constraints, the KKT optimality conditions are 
\begin{align}\label{e5}
\frac{\partial \mathcal{L}}{\partial x} = x-x_\ell' - \mu_1s_\ell w   &=0,\\ \label{e6}
\frac{\partial \mathcal{L}}{\partial w} = w - w_\ell' - \mu_1s_\ell  x - \mu_1s_\ell  \xi -\mu_2 t_\ell  &=0,\\ 
 \label{e7}
\frac{\partial \mathcal{L}}{\partial \xi} = \xi-\xi_\ell' - \mu_1 s_\ell w&=0,\\
\label{e8}
\mu_1\geq 0, \quad \mu_1\left( |y_\ell| - s_\ell xw - s_\ell \xi w \right ) &=0,\\
 \label{e9} \mu_2 \geq 0, \quad \mu_2\left( t_\ell w \right)&=0.
\end{align}
We now proceed with the possible cases.

\textbf{Case 1.} $\mu_1=\mu_2=0$:\\
In this case we have $(x,w,\xi)=(x_\ell',w_\ell',\xi_\ell')$ and this result would only be acceptable when $|y_\ell| - s_\ell x_\ell'w_\ell' - s_\ell \xi_\ell' w_\ell' \leq 0$ and $t_\ell w_\ell'\geq 0$.

\textbf{Case 2.} $\mu_1=0$, $t_\ell w =0$:\\ 
In this case the first feasibility constraint of \eqref{projeq} requires that $|y_\ell|\leq 0$, which is not possible when $|y_\ell|>0$.

\textbf{Case 3.} $|y_\ell| - s_\ell xw - s_\ell \xi w = 0$, $t_\ell w =0$:\\ 
Similar to the previous case, this cannot happen when $|y_\ell|>0$.

\textbf{Case 4.} $\mu_2=0$, $|y_\ell| - s_\ell xw - s_\ell \xi w  =0$:\\ 
In this case we have
\[|y_\ell| = s_\ell xw + s_\ell \xi w.
\] 
Now combining this observation with \eqref{e5} and \eqref{e7} yields 
\begin{align}\label{e18}
|y_\ell| = s_\ell \left(   x_\ell' + \mu_1s_\ell w \right) w + s_\ell \left(   \xi_\ell' + \mu_1s_\ell w \right) w,
\end{align}
and therefore
\begin{align}\label{e17}
\mu_1 = \frac{|y_\ell| - s_\ell \left(   x_\ell' + \xi_\ell' \right)w}{2w^2}.
\end{align}
Similarly, \eqref{e6} yields 
\begin{equation}\label{e19}
w=w_\ell'+ \mu_1 s_\ell \left(   x_\ell' + \mu_1s_\ell w \right)  + \mu_1 s_\ell \left(   \xi_\ell' + \mu_1s_\ell w \right).
\end{equation}
Knowing that $w\neq 0$, $\mu_1$ can be eliminated between \eqref{e18} and \eqref{e19} to generate the following forth order polynomial equation in terms of $w$:
\begin{align*}
2w^4-2w_\ell'w^3 +s_\ell|y_\ell|\left(x_\ell'+\xi_\ell' \right)w- y_\ell^2=0.
\end{align*}
After solving this 4-th order polynomial equation (e.g., the root command in MATLAB) we pick the real root $w$ which obeys
\begin{align}\label{eqconsts}
t_\ell  w\geq 0, \qquad |y_\ell| - s_\ell \left(   x_\ell' + \xi_\ell' \right) w\geq 0.
\end{align}
Note that the second inequality in \eqref{eqconsts} warrants nonnegative values for $\mu_1$ thanks to \eqref{e17}. After picking the right root, we can explicitly obtain $\mu_1$ using \eqref{e19} and calculate the solutions $x$ and $\xi$ using \eqref{e5} and \eqref{e7}.
Technically, in using the ADMM scheme for each $\ell$ we solve a forth-order polynomial equation and find the projection.

\section{Numerical Experiments}\label{experiments}

In this section, we provide numerical experiments on synthetic random data and natural images where the signals follow the model in \eqref{eq:measurements}. We first show a phase portrait that verifies Theorem \ref{thm:Noisy_Main} in the noiseless case. Consider the following measurements: fix $N \in \{20,40,\dots, 300\}$, $L \in \{4,8,\dots,140\}$ and let $K = N$. Let the target signal $(\ho, \mo) \in \mathbb{R}^{K}\times \mathbb{R}^{N}$ be such that both $\ho$ and $\mo$ have $0.05N$ non-zero entries with the nonzero indices randomly selected and set to $\pm 1$. Let $S_1$ and $S_2$ be the number of nonzero entries in $\ho$ and $\mo$, respectively. Let $\mB \in \mathbb{R}^{L\times K}$ and $\mC \in \mathbb{R}^{L\times N}$ such that $B_{ij}\sim \frac{1}{\sqrt{L}}\mathcal{N}(0,1)$ and $C_{ij}\sim \frac{1}{\sqrt{L}}\mathcal{N}(0,1)$. Lastly, let $\vy = \mB\ho\odot \mC\mo$ and $\vt = \text{sign}(\mB\ho)$.

\begin{figure}[H]
\centering
\includegraphics[scale = 0.25]{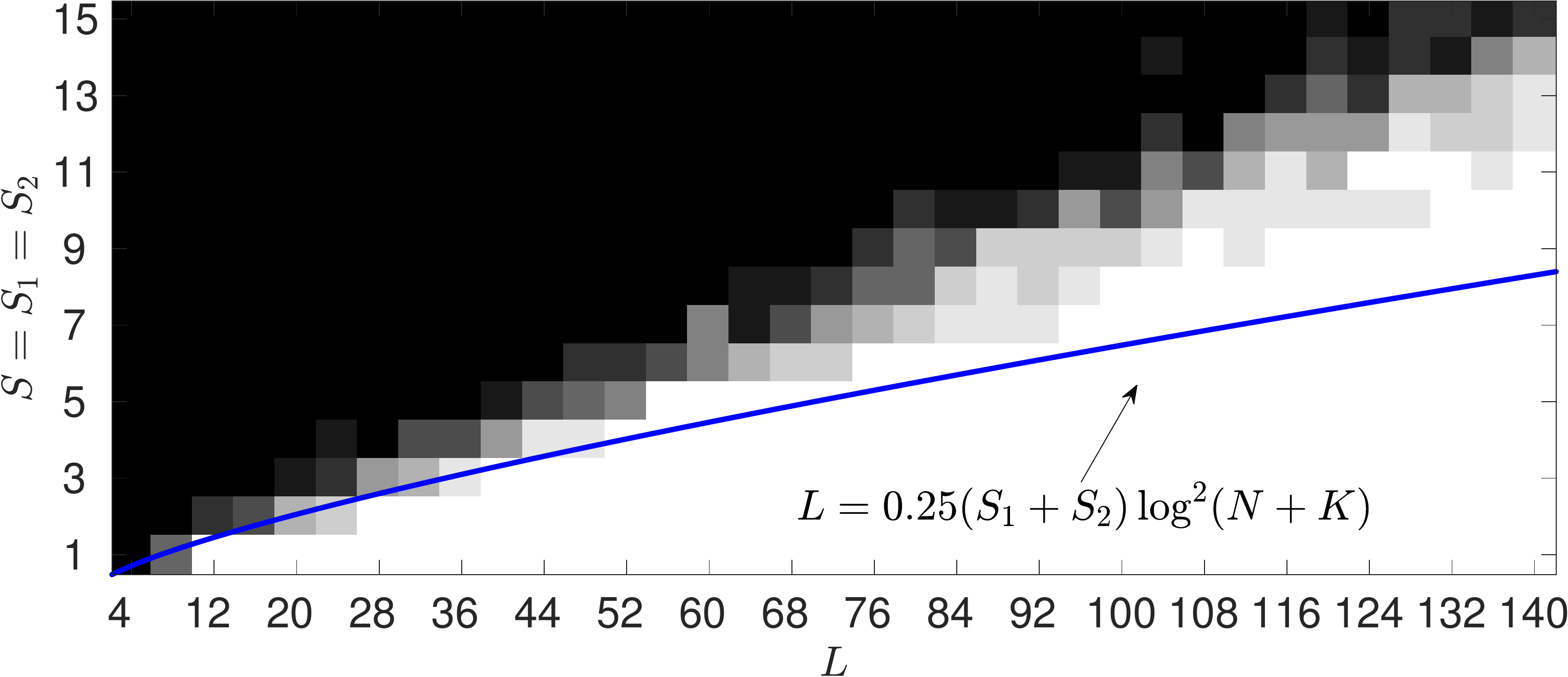}
\caption{The empirical recovery probability from synthetic data with sparsity  level $S$ as a function of total number of measurements $L$. Each block correspond to the average from 10 independent trials. White blocks correspond to successful recovery and black blocks correspond to unsuccessful recovery. The area to the right of the line satisfies $L > 0.25(S_1+S_2)\log^2(N+K)$.}
\label{phase_plot_noiseless}
\end{figure}
Figure \ref{phase_plot_noiseless} shows the fraction of successful recoveries from 10 independent trials using \eqref{eq:BH} for the bilinear inverse problem \eqref{eq:measurements} from data as described above. Let $(\hat{\vh},\hat{\vm})$ be the output of \eqref{eq:BH} and let $(\tilde{\vh},\tilde{\vm})$ be the candidate minimizer. We solve \eqref{eq:BH} using an ADMM implementation similar to the ADMM implementation detailed in Section \ref{algorithm} with the step size parameter $\rho = 1$. For each trial, we say \eqref{eq:BH} successfully recovers the target signal if  $\|(\hat{\vh},\hat{\vm})-(\tilde{\vh},\tilde{\vm})\|_{2} <10^{-10}$. Black squares correspond to no successful recovery and white squares correspond to 100\% successful recovery. The line corresponds to $L = C(S_1+S_2)\log^{2}(K+N)$ with $C = 0.25$ and indicates that the sample complexity constant in Theorem \ref{thm:Noisy_Main}, in the noiseless case, is not very large. 

We now show the result of using the total variation BranchHull program \eqref{eq:TVBH} to remove distortions from real images $\tilde{\vy} \in \R^{p\times q}$. In the experiments, The observation $\vy \in \mathbb{R}^L$ is the column-wise vectorization of the image $\tilde{\vy}$, the target signal $\vw=\mB\vh$ is the vectorization of the piecewise constant image and $\vx=\mC\vm$ corresponds to the distortions in the image. We use \eqref{eq:TVBH} to recover piecewise constant target images like in the foreground of Figure \ref{data_y} with TV$(\mB\vh) =\|\mD \mB\vh\|_1$, where $\mD = \left[\begin{array}{c}\mtx{D}_v \\ \mtx{D}_h\end{array}\right]$ in block form. Here, $\mD_v \in \mathbb{R}^{(L-q)\times L}$ and  $\mD_h \in \mathbb{R}^{(L-p)\times L}$ with
\[(\mD_v)_{ij} = \left\{\begin{array}{c l} -1 & \text{if } j=i+\left(\frac{i-1}{p-1}\right)_- \\ 1 & \text{if } j=i+1+\left(\frac{i-1}{p-1}\right)_-\\ 0 & \text{otherwise}\end{array} \right., \ (\mD_h)_{ij} = \left\{\begin{array}{c l} -1 & \text{if } j=i\\ 1 & \text{if } j=i+p\\ 0 & \text{otherwise}\end{array} \right. .
\]

Lastly, we solve \eqref{eq:TVBH} using the ADMM algorithm detailed in Section \ref{algorithm} with $\mP=\mD\mB$.

\begin{figure}[H]
\centering
\begin{tabular}[c]{lc}
	\begin{subfigure}{.24\textwidth}
		\captionsetup{skip=5pt,oneside,margin={0cm,0cm}} 
		\includegraphics[scale = .16]{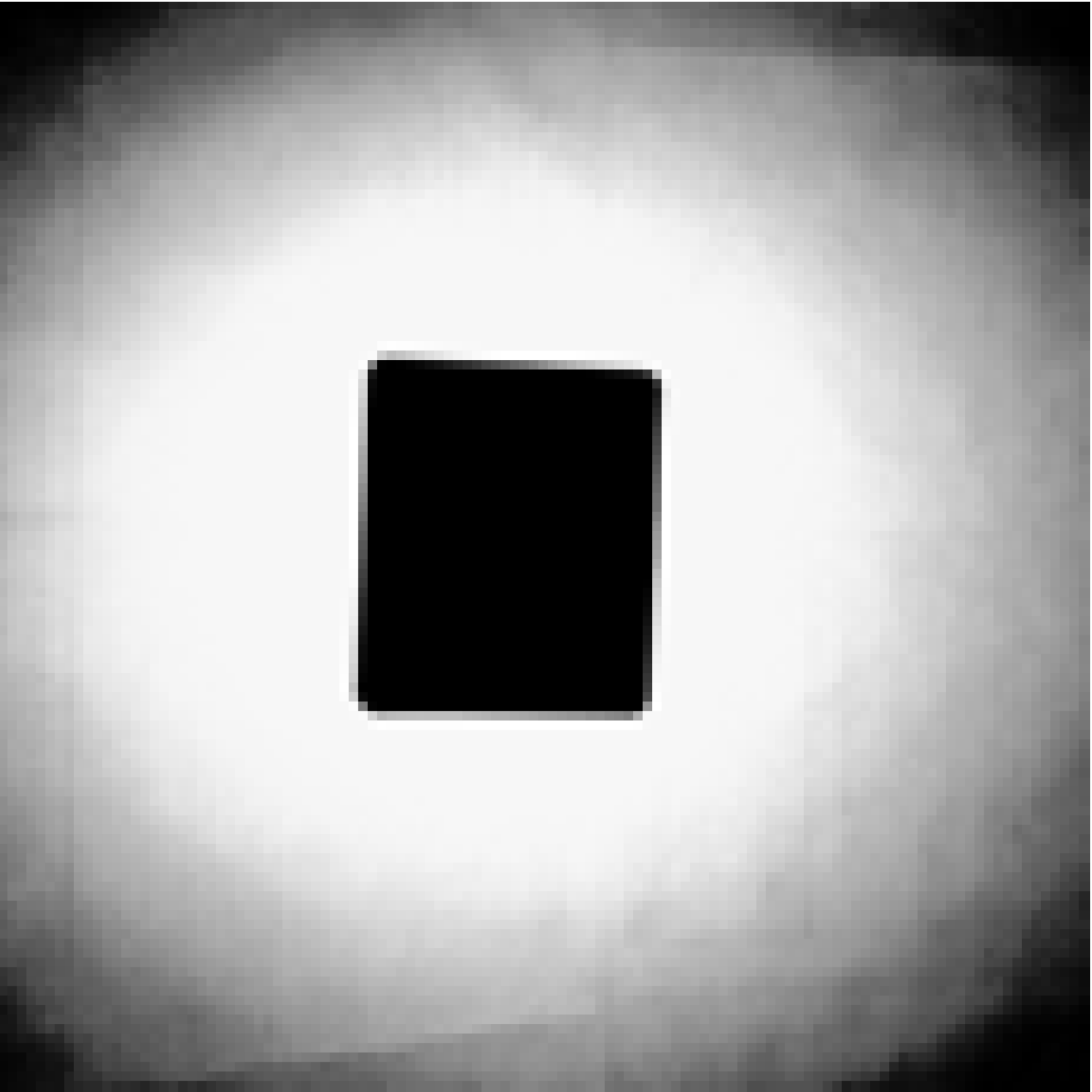}
		\caption{Distorted image}
		\label{data_y}
	\end{subfigure}&
	\begin{subfigure}{.24\textwidth}
		\captionsetup{skip=5pt,oneside,margin={-0cm,0cm}} 
		\centering
		\includegraphics[scale = .16,frame]{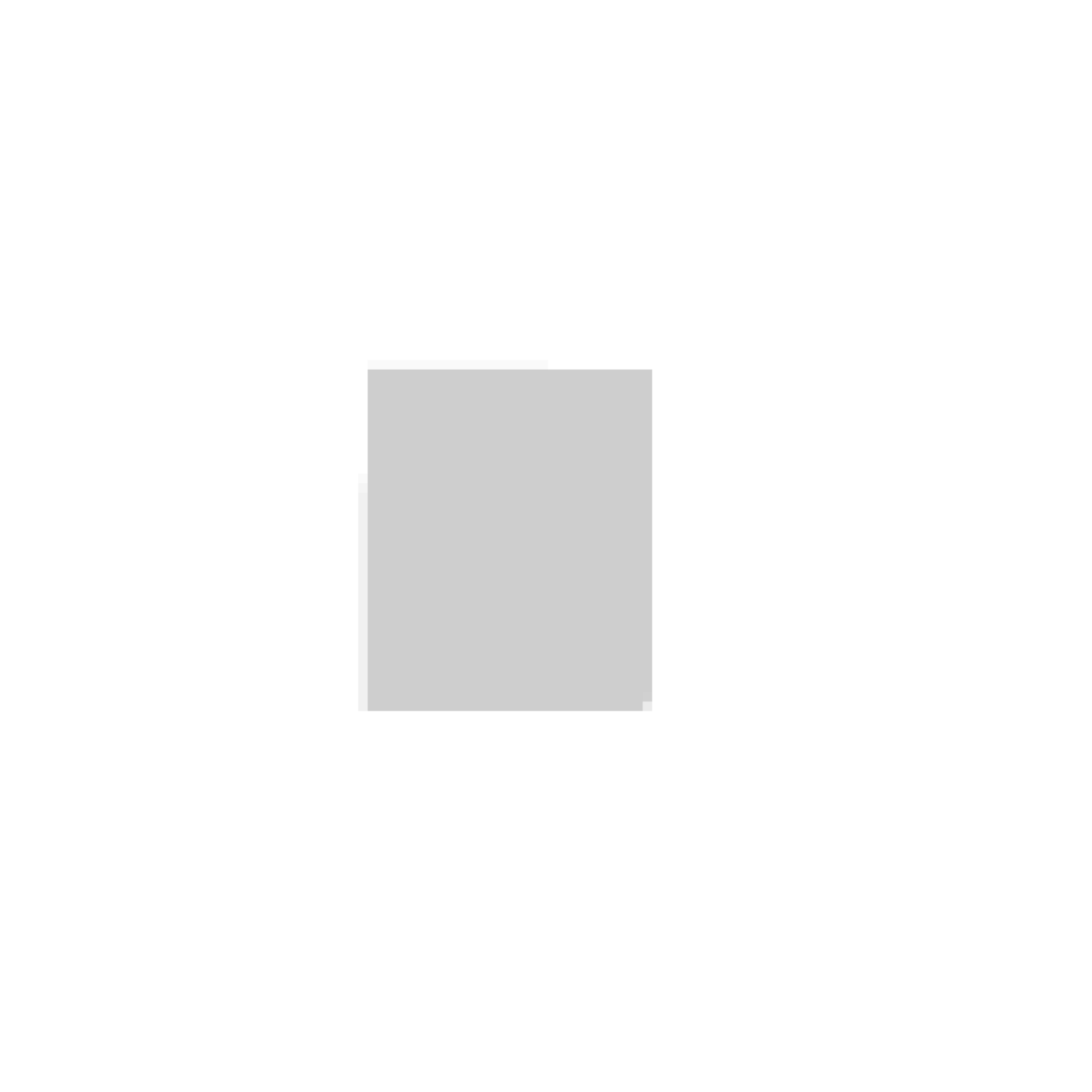} 
		\caption{Recovered image}
		\label{output_h}
	\end{subfigure}\\
	\begin{subfigure}{.24\textwidth}
		\captionsetup{skip=5pt,oneside,margin={0cm,0cm}} 
		\includegraphics[scale = .16]{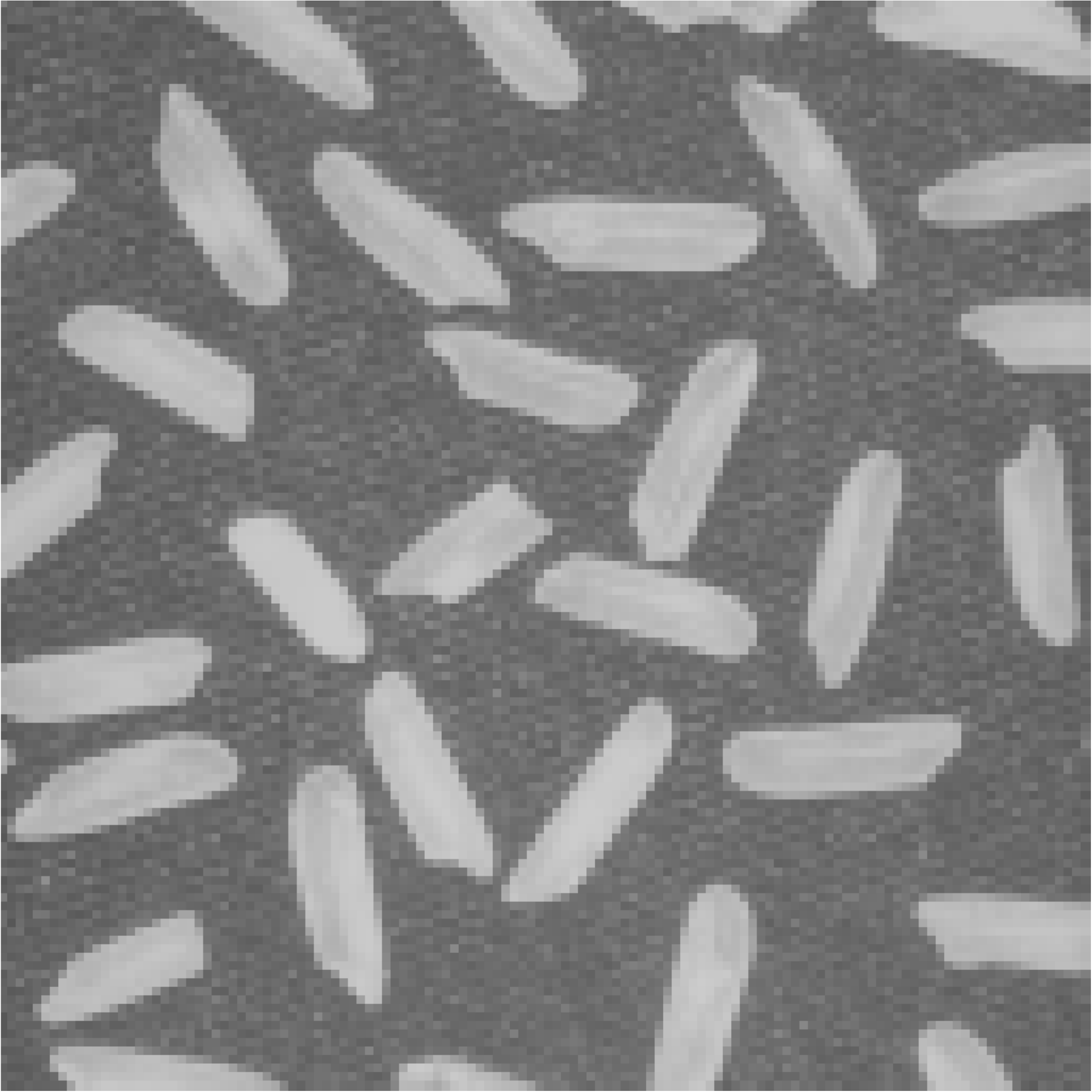}
		\caption{Distorted image}
		\label{rice_y}
	\end{subfigure}&
	\begin{subfigure}{.24 \textwidth}
		\captionsetup{skip=5pt,oneside,margin={-0cm,0cm}} 
		\centering
		\includegraphics[scale = .16]{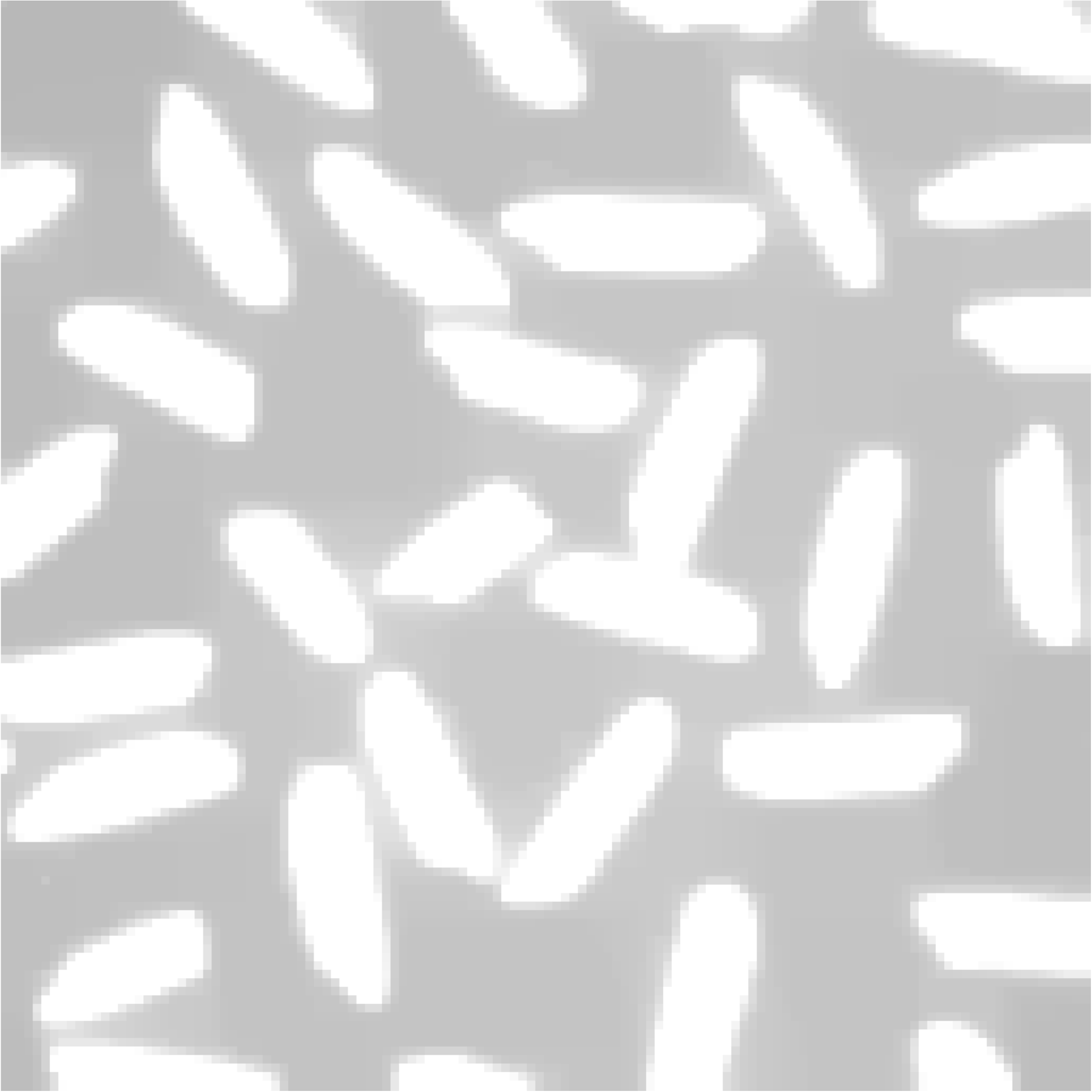} 
		\caption{Recovered image}
		\label{rice_h}
	\end{subfigure}
\end{tabular}
	\caption{Panel (a) shows an image of a mousepad with distortions and panel(b) is the piecewise constant image recovered using total variation $\ell_1$-BranchHull. Similarly, panel (d) shows an image containing rice grains and panel (e) is the recovered image.}
\end{figure}
\vspace{-10pt}

 We use the total variation BranchHull program on two real images. The first image, shown in Figure \ref{data_y}, was captured using a camera and resized to a $115 \times 115$ image. The measurement $\vy \in \mathbb{R}^{L}$ is the vectorization of the image with $L = 13225$. Let $\mB$ be the $L \times L$ identity matrix. Let $\mF$ be the $L \times L$ inverse DCT matrix. Let $\mC \in \mathbb{R}^{L \times 300}$ with the first column set to $\bf{1}$ and remaining columns randomly selected from columns of $\mF$ without replacement. The matrix $\mC$ is scaled so that $\|\mC\|_F = \|\mB\|_F = \sqrt{L}$. The vector of known sign $\vt$ is set to $\bf{1}$. Let $(\hat{\vh},\hat{\vm},\hat{\vxi})$ be the output of \eqref{eq:TVBH} with $\lambda = 10^3$ and $\rho = 10^{-4}$. Figure \ref{output_h} corresponds to $\mB\hat{\vh}$ and shows that the object in the center was successfully recovered.

The second real image, shown in Figure \ref{rice_y}, is an image of rice grains. The size of the image is $128 \times 128$. The measurement $\vy \in \mathbb{R}^{L}$ is the vectorization of the image with $L = 16384$. Let $\mB$ be the $L \times L$ identity matrix. Let $\mC \in \mathbb{R}^{L \times 50}$ with the first column set to $\bf{1}$. The remaining columns of $\mC$ are sampled from Bessel function of the first kind $J_{\nu}(\gamma)$ with each column corresponding to a fixed $\gamma \in \mathbb{R} $. Specifically, fix $\vg \in \mathbb{R}^{L}$ with $g_i = -9+14\frac{i-1}{L-1}$. For each remaining column $\vc$ of $\mC$, fix $\vzeta \sim \mathcal{N}(\mtx{0},\text{I}_3)$ and let $c_i = J_{\frac{g_i}{6+0.1|\zeta_1|}+5|\zeta_2|}(0.1+10|\zeta_3|)$. The matrix $\mC$ is scaled so that $\|\mC\|_F = \|\mB\|_F =\sqrt{L}$. The vector of known sign $\vt$ is set to $\bf{1}$. Let $(\hat{\vh},\hat{\vm},\hat{\vxi})$ be the output of \eqref{eq:TVBH} with $\lambda = 10^3$ and $\rho = 10^{-7}$. Figure \ref{rice_h} corresponds to $\mB\hat{\vh}$. 

\section{Proof Outline}\label{proof}
{\change In this section, we provide a proof of Theorem \ref{thm:Noisy_Main} by considering a program similar to $\l_1$-BrachHull program \eqref{eq:BH} with a different representation of the constraint set. Let $(\hat{\wl},\hat{\xl}) = (\blt\hath,\clt\hatm)$, $\alpha_\l = \left(|\hat{\wl}|+|\hat{\xl}|\right)/2$ and define a convex function
\begin{align*}
f(\wl,\xl) = \gamma(\wl,\xl)\bigg(\sqrt{4|\yl| + (\wl-s_\ell\xl)^2}-t_\ell(\wl+s_\ell\xl)\bigg)
\end{align*}
where $\gamma_\l:\R^{2}\rightarrow\R_{>0}$ is a piecewise constant function such that 
\begin{align}\label{defineconvexf}
	\gamma(\wl,\xl) = \left\{
	\begin{array}{ll}
		1, \quad \text{if } 	\alpha_\l >1 \text{ and } \sqrt{4|\yl| + (\wl-s_\ell\xl)^2}-t_\ell(\wl+s_\ell\xl)\leq 0\\
		\alpha_\l, \quad \text{otherwise.}
	\end{array}
	\right..
\end{align}
Figures \ref{fig:convex_levelsetless} and \ref{fig:convex_levelsetmore} show typical $f(\wl,\xl)$ for $\alpha_\l \leq 1$ and $\alpha >1$, respectively. 

Let $f_\l(\vh,\vm) =f(\blt\vh,\clt\vm)$ with $\gamma_\l(\vh,\vm):= \gamma(\blt\vh,\clt\vm)$. We note that $f_\l:\R^{K+N}\rightarrow\R$ is also a convex function  because its epigraph is a convex set. The epigraph is convex because it is the inverse image of a convex set over a linear map. Define a one-sided loss function
\[
\mathcal{L}(\vh,\vm) :=  \tfrac{1}{L}\sum_{\ell = 1}^L \left[f_\ell(\vh,\vm)\right]_+,
\]
where $[\ \cdot \ ]_+$ denotes the positive side. We analyze the following generalized version of the $\ell_1$-BranchHull program:
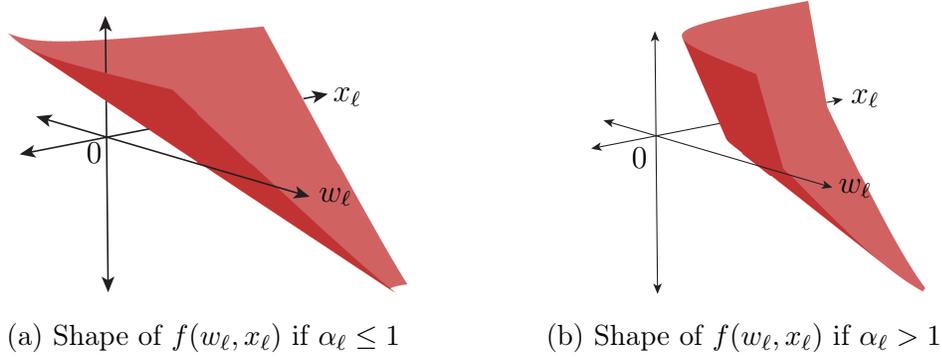
\begin{figure}[H]
\floatsetup[subfigure]{captionskip=0pt} 
\ffigbox[]{\hspace{0pt}
    \begin{subfloatrow}[2]
      \ffigbox[]{%
      \begin{overpic}[scale=.6]{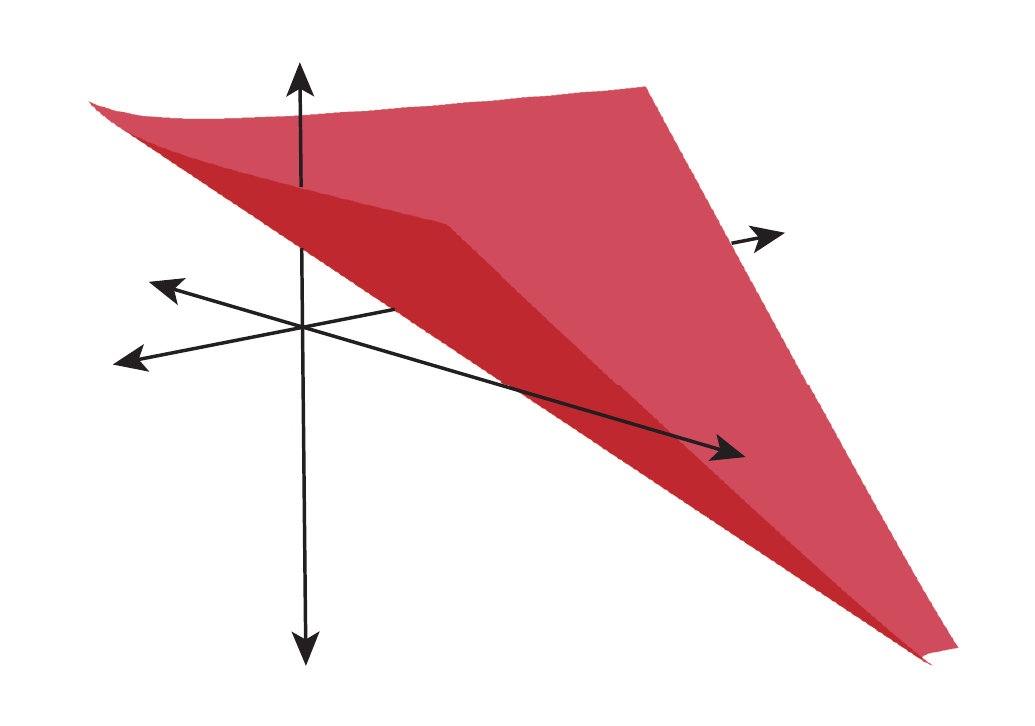}
     	\put(25,33){$0$}
		\put(74,25){$w_\ell$}
		\put(77,46){$x_\ell$}
	  \end{overpic}
        }{\caption{Shape of $f(\wl,\xl)$ if $\alpha_\l \leq 1$}\label{fig:convex_levelsetless}}
      \hspace{\fill}\ffigbox[\FBwidth]{\raisebox{0cm}{
        \begin{overpic}[scale=.4]{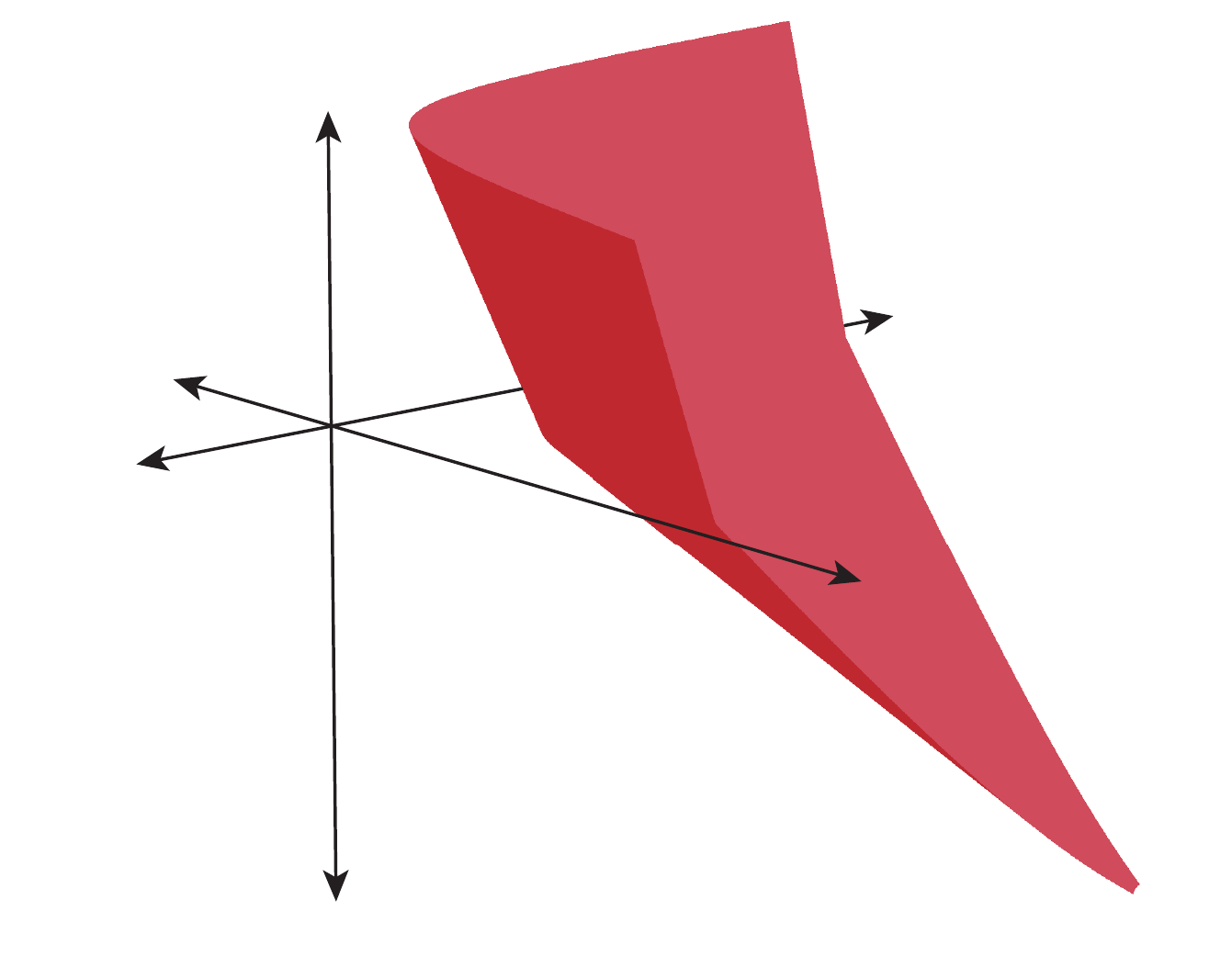}
     	\put(21,37){$0$}
		\put(71.5,31){$w_\ell$}
		\put(74.5,53){$x_\ell$}
		\end{overpic}}}{\caption{Shape of $f(\wl,\xl)$ if $\alpha_\l > 1$}\label{fig:convex_levelsetmore}}
    \end{subfloatrow}
}{\caption{Panels (a) and (b) shows the shape of the convex function $f(\wl,\xl)$ for $\alpha_\l \leq 1$ and $\alpha_\l >1$, respectively. When $\alpha_\l \leq 1$, $f(\wl,\xl)$ is differentiable everywhere. When $\alpha_\l >1$, $f(\wl,\xl)$ is not differentiable at $(\wl,\xl)$ with $f(\wl,\xl)=0$.}
}
\end{figure}
\begin{align}
\text{$\ell_1$-GBH}: &\underset{\vh\in \R^K, \vm \in \R^N}{\minimize}~\|\vh\|_1+\|\vm\|_1\label{eq:GBH}\\ &\text{subject to}~~ \mathcal{L}(\vh,\vm)\leq 0\notag.
\end{align}
Program \eqref{eq:GBH} is equivalent to the $\ell_1$-BranchHull in the sense that the objective and the constraint set of both the programs are the same. Lemma \ref{lem:levelset} shows that the set defined by constraints $\sl(\blt \vh \clt\vm)\geq|\yl|$ with $t_\ell\blt\vh \geq 0$ and the set defined by constraints $f_\ell(\vh,\vm)\leq 0$ are the same set.
\begin{lemma}\label{lem:levelset}
	Fix $(\ho,\mo)\in \R^{K+N}$ such that $\ho \neq 0$ and $\mo \neq 0$. Let $\mB \in \R^{L\times K}$, $\mC\in\R^{L\times N}$ and $\vxi \in \R^L$. Let $\vy\in\R^L$ contain measurements that satisfy \eqref{eq:measurements}. The set $\{(\vh,\vm)\in\R^{K+N}:\sl(\blt \vh \clt\vm)\geq|\yl|, \ t_\ell\blt\vh \geq 0,\ \l\in[L]\}$ is equal to the set $\{(\vh,\vm)\in\R^{K+N}:\mathcal{L}(\vh,\vm)\leq0 \}.$
\end{lemma}
\begin{proof}
Fix an $\l \in \{1,\dots,L\}$. It is sufficient to show that the set $\mathcal{S}_{\ell,1}:=\{(\vh,\vm)\in\R^{K+N}:\sl(\blt \vh \clt\vm)\geq|\yl|, \ t_\ell\blt\vh \geq 0\}$ is equal to the set $\mathcal{S}_{\ell,2}:=\{(\vh,\vm)\in\R^{K+N}:f_\l(\vh,\vm)\leq 0 \}$. Consider a $(\vh,\vm) \in \mathcal{S}_{\l,1}$. We have
\begin{align}
	&|\yl|\leq\sl\blt \vh \clt\vm,\notag\\
	\Leftrightarrow & 4|\yl| + (\blt\vh-\sl\clt\vm)^2\leq (\blt\vh+\sl\clt\vm)^2,\notag\\
	\Leftrightarrow & \sqrt{4|\yl| + (\blt\vh-\sl\clt\vm)^2}\leq t_\l(\blt\vh+\sl\clt\vm),\label{levelset1}\\
	\Leftrightarrow & \gamma_\l(\vh,\vm)\left(\sqrt{4|\yl| + (\blt\vh-\sl\clt\vm)^2}-t_\l(\blt\vh+\sl\clt\vm)\right)\leq0\label{levelset2},
\end{align}
where \eqref{levelset1} holds because $t_\l\blt\vh\geq 0$ and $t_\l\sl\clt\vm \geq 0$ and \eqref{levelset2} holds because $\gamma_\l(\vh,\vm)\geq 0$. Thus, $\mathcal{S}_{\l,1}\subseteq\mathcal{S}_{\l,2}$. Now, consider a $(\vh,\vm) \in \mathcal{S}_{\l,2}$. W.L.O.G. assume $\gamma_\l(\vh,\vm)\neq 0$. The reverse implications above implies $|\yl|\leq\sl\blt \vh \clt\vm$. Also, $t_\l(\blt\vh+\sl\clt\vm)\geq 0$ because $4|\yl| + (\blt\vh-\sl\clt\vm)^2\geq 0$. If $\blt\vh = 0$, then $t_\l\blt\vh\geq 0$. If $\blt\vh\neq0$, then
\begin{align*}
	t_\l\blt\vh\left(1+\frac{\sl\clt\vm}{\blt\vh}\right) = t_\l(\blt\vh+\sl\clt\vm) \geq 0.
\end{align*}
So, $t_\l\blt\vh \geq 0$ because $\sl\blt \vh \clt\vm\geq|\yl|$ implies $\left(1+\frac{\sl\clt\vm}{\blt\vh}\right)\geq 0$. Thus, $\mathcal{S}_{\l,2}\subseteq\mathcal{S}_{\l,1}$ as well, which proves that $\mathcal{S}_{\l,1}=\mathcal{S}_{\l,2}$
\end{proof}

 We will first show that if the noise $\vxi$ in the problem statement \eqref{eq:measurements} satisfy $\xi_\l \in [-1,0]$ for all $\l \in \{1,\dots,L\}$, then the $\ell_1$-Generalized BranchHull program \eqref{eq:GBH} recovers a point close to the set $\{(c\hath,c^{-1}\hatm)|c>0\}$. We then extend the result to noise that satisfy condition \eqref{noise-condition}. Since the $\l_1$-Generalized BranchHull program \eqref{eq:GBH} and the $\l_1$-BranchHull program \eqref{eq:BH} are equivalent, the minimizer of $\ell_1$-BranchHull program is then also close to the set $\{(c\hath,c^{-1}\hatm)|c>0\}$.} Our strategy will be to show that for any feasible perturbation $(\dh,\dm) \in \mathbb{R}^{K+N}$, the objective of the $\ell_1$-Generalized BranchHull program \eqref{eq:GBH} strictly increases outside a curved cylinder of radius $r$ centered at the bilinear ambiguity curve $\{(c\hath,c^{-1}\hatm)|c>0\}$, where the radius depends on the level of noise. 

The subgradient of the $\ell_1$-norm at $(\hath,\hatm)$ is
\begin{align*} 
\partial \|(\hath,\hatm)\|_1 := \{&\vg\in \R^{K+N}: \|\vg\|_\infty \leq 1 ~\text{and}~ \vg_{\Gamma_h} =\sign(\hath_{\Gamma_h}) \ ,\vg_{\Gamma_m} = \text{sign}(\hatm_{\Gamma_m})\},
\end{align*}
where $\Gamma_h,$ and $\Gamma_m$ denote the support of non-zeros in $\hath$, and $\hatm$, respectively. We first consider the following descent direction that are orthogonal to the set $\mathcal{N} :=\text{span}(-\hath,\hatm)$ 
\begin{align}\label{eq:setD}
&\left\{(\dh,\dm) \in \mathcal{N}_\perp:\begin{array}{l}
	\big\<\vg,(\dh,\dm)\big\> \leq 0,\notag\\[\smallskipamount] \forall \vg \in \partial \|(\hath,\hatm)\|_1
\end{array}\right\}\notag\\
\subseteq & \left\{(\dh,\dm) \in \mathcal{N}_\perp: \begin{array}{l}\<\vg_{\Gamma_h},\dh_{\Gamma_h}\>+\<\vg_{\Gamma_m},\dm_{\Gamma_m}\>\\[\smallskipamount] +\|(\dh_{\Gamma_h^c},\dm_{\Gamma_m^c})\|_1 \leq 0\end{array}\right\}\notag\\
\subseteq &\left\{ (\dh,\dm) \in \mathcal{N}_\perp:\begin{array}{l} \|(\dh_{\Gamma_h^c},\dm_{\Gamma_m^c})\|_1\leq \\[\smallskipamount]\|\vg_{\Gamma_h \cup \Gamma_m}\|_2\|(\dh_{\Gamma_h},\dm_{\Gamma_m})\|_2\end{array}\right\}\notag\\
= &\left\{ (\dh,\dm) \in \mathcal{N}_\perp: \begin{array}{l}\|(\dh_{\Gamma_h^c},\dm_{\Gamma_m^c})\|_1 \leq \\[\smallskipamount] \sqrt{S_1+S_2}\|(\dh_{\Gamma_h},\dm_{\Gamma_m})\|_2\end{array}\right\}=: \mathcal{D}
\end{align}
and show that descent direction from $(\hath,\hatm)\oplus\mathcal{N}$ of large $\l_2$ norm is not feasible in \eqref{eq:GBH}. We do this by quantifying the ``width" of the set $\setD$ through a Rademacher complexity, and a probability that one of the subgradients of the constraint functions lie in a certain half space. In the noiseless case, we show that the solution of \eqref{eq:GBH} is in the set $(\hath,\hatm)\oplus\mathcal{N}$. Since the only point in the set $(\hath,\hatm)\oplus\mathcal{N}$ consistent with the constraint of \eqref{eq:GBH} is $(\hath,\hatm)$, the minimizer of \eqref{eq:GBH}, in the noiseless case, is then $(\hath,\hatm)$. In the noisy case, we use the boundedness of the feasible directions from the line $(\hath,\hatm)\oplus\mathcal{N}$ along with the observation that the feasible hyperbolic set diverges away from $(\hath,\hatm)\oplus\mathcal{N}$ to conclude the solution of minimizer of \eqref{eq:GBH} is close to the bilinear ambiguity curve $\{(c\hath,c^{-1}\hatm)|c>0\}$.

{\change Recall that the constraint set of the $\ell_1$-BranchHull \eqref{eq:BH} has $L$ hyperbolic constraints and $L$  linear constraints. Let $\hat{\vy} = \mB\hath\odot\mC\hatm$ be the noiseless data. Using Lemma \ref{lem:levelset}, each pair of hyperbolic and linear constraint $\{(\vh,\vm): \sl\blt\vh\clt\vm \geq|\hat{y}_\l|,\ t_{\ell}\blt\vh\geq 0\}$ can be expressed as $\hat{f}_\ell(\vh,\vm) \leq 0$ where
\begin{align*}
\hat{f}_\ell(\vh,\vm) = \gamma_\ell(\vh,\vm) \bigg(&\sqrt{4|\hat{y}_\l| + (\blt\vh-s_\ell\clt\vm)^2}-t_\ell(\blt\vh+s_\ell\clt\vm)\bigg).
\end{align*}
First we note that a subgradient of $\hat{f}_\l(\vh,\vm)$ at $(\hath,\hatm)$ is $\vz_\l := -\sl(\clt\hatm\blt,\blt\hath\clt)$. To see this, recall that $\alpha_\l = \left(\blt\hath+\clt\hatm\right)/2$ and let
\[
\hat{g}_\l(\vh,\vm) = \alpha_\ell \bigg(\sqrt{4|\hat{y}_\l| + (\blt\vh-s_\ell\clt\vm)^2}-t_\ell(\blt\vh+s_\ell\clt\vm)\bigg).
\]
When $\alpha_\l \leq 1$, we have $\hat{f}_\l = \hat{g}_\l$ because $\gamma_\l(\vh,\vm) = \alpha_\l$ by definition of $\gamma(\blt\vh,\clt\vm)=\gamma_\l(\vh,\vm)$ in \eqref{defineconvexf}. When $\alpha_\l > 1$, $\hat{f}_\l$ is non-differentiable at $(\vh,\vm)$ where $\hat{f}_\l(\vh,\vm)=0$. Figure \ref{fig:convex_levelsetmore} shows the shape of $\hat{f}_\l$ when $\alpha_\l >1$. In this case, we have $\hat{f}_\l(\hath,\hatm) = 0$ and $\hat{f}_\l(\vh,\vm) = \hat{g}_\l(\vh,\vm)$ for $(\vh,\vm)$ that satisfy $\hat{f}_\l(\vh,\vm)\leq 0$. Thus, $\nabla\hat{g}_\l(\hath,\hatm) \in \partial\hat{f}_\l(\hath,\hatm)$ and $-\sl(\clt\hatm\blt,\blt\hath\clt)\in \partial\hat{f}_\l(\hath,\hatm)$ if $\nabla\hat{g}_\l(\hath,\hatm)=-\sl(\clt\hatm\blt,\blt\hath\clt)$. So, consider 
\begin{align*}
	\frac{\partial \hat{g}_\l}{\partial \vh}(\hath,\hatm) = &\alpha_\l\left(\frac{2(\blt\hath-s_\l\clt\hatm)\bl}{2\sqrt{4|\hat{y}_\ell| + (\blt\hath-s_\ell\clt\hatm)^2}}-t_\l\bl\right)\\
	 = &\alpha_\l\left(\frac{(\blt\hath-s_\l\clt\hatm)\bl}{\sqrt{(\blt\hath+s_\ell\clt\hatm)^2}}-t_\l\bl\right)\\
	 = &\alpha_\l\left(\frac{(\blt\hath-s_\l\clt\hatm)\bl}{|t_\l(\blt\hath+s_\ell\clt\hatm)|}-t_\l\bl\right)\\
	 = &\alpha_\l\left(\frac{(\blt\hath-s_\l\clt\hatm)\bl-(\blt\hath+s_\ell\clt\hatm)\bl}{t_\l(\blt\hath+s_\ell\clt\hatm)}\right)\\
	 = &\alpha_\l\left(\frac{(\blt\hath-s_\l\clt\hatm)\bl-(\blt\hath+s_\ell\clt\hatm)\bl}{t_\l(\blt\hath+s_\ell\clt\hatm)}\right)\\
	 = &-s_\l\clt\hatm\bl
	\end{align*} 
	where fourth equality holds because $t_\l\blt\hath\geq0$ and $t_\l s_\l\clt\hatm\geq0$ and the last equality holds because $\alpha_\l = t_\ell(\blt\hath+\sl\clt\hatm)/2$. Similarly, we have $\frac{\partial\hat{g_\l}}{\partial\vm}(\hath,\hatm) = -s_\l\blt\hath\cl$ and $\nabla\hat{g}_\l(\hath,\hatm)=-\sl(\clt\hatm\blt,\blt\hath\clt)$.
}	
Define the Rademacher complexity of a set $\setD \subset \R^{K+N}$ as
\begin{align}\label{eq:Rademacher-Complexity}
\mathfrak{C}(\setD) := \E\sup_{(\vh,\vm)\in \setD} \tfrac{1}{\sqrt{L}}\sum_{\ell=1}^L\varepsilon_\ell \left\<\vz_\l,\tfrac{ (\vh,\vm)}{\|(\vh,\vm)\|_2}\right\>,
\end{align}
where $\varepsilon_1,\varepsilon_2, \ldots, \varepsilon_L$ are iid Rademacher random variables independent of everything else. For a set $\setD$, the quantity $\mathfrak{C}(\setD)$ is a measure of width of $\setD$ around the origin in terms of the subgradients of the constraint functions. Our results also depend on a probability $p_{\tau}(\setD)$, and a positive parameter $\tau$ introduced below
\begin{align}\label{eq:Tail-Prob}
\mathfrak{p}_{\tau}(\setD) = \inf_{(\vh,\vm)\in \mathcal{D}} \PP\left(\left\<\vz_\ell,\tfrac{ (\vh,\vm)}{\|(\vh,\vm)\|_2}\right\>\geq \tau\right).
\end{align}
Intuitively, $p_{\tau}(\setD)$ quantifies the size of $\setD$ through the subgradient vector. For a small enough fixed parameter, a small value of $p_{\tau}(\setD)$ means that the $\setD$ is mainly invisible to the subgradient vector. 

{We now state a lemma which shows that if the noise $\vxi$ is such that 
\begin{equation}\label{noise_sided}
	\xi_{\l'} = 0 \text{, for some } \l' \in \{1,\dots,L\}, \text{ with } \xi_\l\in\left[-1,0\right] \text{ for all }\l \in \{1,\dots,L\},
\end{equation}
  and the number of measurements $L \geq \left(\frac{2\mathfrak{C}(\mathcal{D})+t\tau}{\tau \mathfrak{p}_{\tau}(\mathcal{D})}\right)^2$ for any $t>0$, then the solution of \eqref{eq:GBH} is close to the bilinear ambiguity curve $\{(c\hath,c^{-1}\hatm)|c>0\}$}. The Proof of this lemma is based on small ball method developed in \cite{koltchinskii2015bounding,mendelson2014learning} and further studied in  \cite{lecue2018regularization,lecue2017regularization, bahmani2017anchored}.

\begin{lemma}\label{lem:Mendelson}
	Let $\setD$ be the set of descent directions, already characterized in \eqref{eq:setD}, for which $\mathfrak{C}(\setD)$, and $\mathfrak{p}_{\tau}(\setD)$  can be determined using  \eqref{eq:Rademacher-Complexity}, and \eqref{eq:Tail-Prob}. Let the noise $\vxi \in \mathbb{R}^L$ be such that \eqref{noise_sided} is satisfied. Choose $ L \geq \left(\frac{2\mathfrak{C}(\mathcal{D})+t\tau}{\tau \mathfrak{p}_{\tau}(\mathcal{D})}\right)^2$ for any $t >0$. Then the solution $ (\tildeh,\tildem) $  of the BH in \eqref{eq:BH} satisfies
	\begin{align*}
		&\left\|(\tildeh,\tildem)-\left(c\hath ,c^{-1}\hatm\right)\ \right\|_2\leq 36\sqrt{\|\hath\|_2\|\hatm\|_2}\sqrt{\|\vxi\|_\infty}
	\end{align*}
for some $\c>0$ with probability at least $1-\mathrm{e}^{-2Lt^2}$. Furthermore, $(\tildeh,\tildem)=(\hath,\hatm)$ if $\vxi = \vzero$. 
\end{lemma}

\begin{proof}
{\change Without loss of generality, we analyze the $\l_1$-Generalized BranchHull program \eqref{eq:GBH}. We note that $(\hath,\hatm)$ is feasible in \eqref{eq:GBH} because the noise satisfy $\vxi_\l \in [-1,0]$ for all $\l \in \{1,\dots,L\}$. We first control the set of feasible descent direction $(\dh,\dm) \in \mathcal{D}$ from the set $(\hath,\hatm)\oplus\mathcal{N}$. Since  $(\dh,\dm)$ is a feasible perturbation from a point $(\vh^*,\vm^*):=\left((1-\beta)\hath,(1+\beta)\hatm\right)$ for some $\beta \in \R$, we have from \eqref{eq:GBH}
\begin{align}\label{eq:interim-eq7}
\mathcal{L}(\vh^*+\dh,\vm^*+\dm) \leq 0.
\end{align}
Note that because of \eqref{eq:interim-eq7}, $f_\l(\vh^*+\dh,\vm^*+\dm) \leq 0$ for all $\l \in \{1,\dots,L\}$ since, by definition, $\mathcal{L}(\vh,\vm) = \tfrac{1}{L}\sum_{\l=1}^{L}\left[f_\l(\vh,\vm)\right]_+$. Thus, $\gamma_\l(\vh^*+\dh,\vm^*+\dm)$ in $f_\l(\vh^*+\dh,\vm^*+\dm)$ satisfies $0\leq\gamma_\l(\vh^*+\dh,\vm^*+\dm) \leq 1$. We now expand the loss function $\mathcal{L}(\vh,\vm)$ at $(\vh^*+\dh,\vm^*+\dm)$. Consider
\begin{align}
&\big[f_\ell(\vh^*+\dh,\vm^*+\dm)\big]_+\notag\\
=&\bigg[\gamma_\ell\bigg(\sqrt{\bigg(\left(\blt(\vh^*+\dh)-\sl\clt(\vm^*+\dm)\right)^2+4\sl\hat{y}_\ell+4\sl\hat{y}_\ell\xi_\ell\bigg)}\notag\\
&-t_\ell\bigg(\blt(\vh^*+\dh)+\sl\clt(\vm^*+\dm)\bigg)\bigg)\bigg]_+\notag\\
\geq & \bigg[\gamma_\ell\bigg(\sqrt{4|\hat{y}_\ell|+\left(\blt(\vh^*+\dh)-\sl\clt(\vm^*+\dm)\right)^2}-t_\ell\left(\blt(\vh^*+\dh)+\sl\clt(\vm^*+\dm)\right)\notag\\
&-\sqrt{\left[-4\sl\hat{y}_\ell\xi_\ell\right]_+}\bigg)\bigg]_+\label{eq:interim-eq8}\\
\geq & \left[\hat{f}_\ell(\hath-\beta\hath+\dh,\hatm+\beta\hatm+\dm)\right]_+-2\gamma_\ell\sqrt{\left[-\sl\hat{y}_\ell\xi_\ell\right]_+}\notag\\
\geq & \left[\langle \vz_\ell,(-\beta\hath+\dh,\beta\hatm+\dm)\rangle\right]_+-2\sqrt{\left[-\sl\hat{y}_\ell\xi_\ell\right]_+}\label{eq:interim-eq9},\\
= & \left[\langle \vz_\ell,(\dh,\dm)\rangle\right]_+-2\sqrt{\left[-\sl\hat{y}_\ell\xi_\ell\right]_+}\label{eq:interim-eq1},
\end{align}
where in \eqref{eq:interim-eq8} we use $\sign(\yl) = \sign(\hat{y}_\ell)$ along with the fact that for $a\geq 0$ and $b<0$ with $a+b\geq0$, we have $\sqrt{a+b} \geq \sqrt{a}-\sqrt{-b}$ and for $a\geq0$ and $b\geq0$, we have $\sqrt{a+b}\geq\sqrt{a}$. Also, \eqref{eq:interim-eq9} holds because $\hat{f}_\l$ is convex with $\vz_\l \in \partial\hat{f}_\l(\hath,\hatm)$ and $0\leq\gamma_\l \leq 1$. Lastly, \eqref{eq:interim-eq1} holds because $\langle \vz_\l,(-\hath,\hatm) \rangle = 0$. Combining \eqref{eq:interim-eq7} and \eqref{eq:interim-eq1}, we get
\begin{align}
	\underbrace{\tfrac{1}{L}\sum_{\ell=1}^L\left[\langle \vz_\l,(\dh,\dm)\rangle\right]_+}_{I}\leq & \tfrac{2}{L}\sum_{\ell=1}^L\sqrt{|\hat{y}_\ell\xi_\ell|}\notag\\
	\leq & \tfrac{2}{L}\sum_{\l=1}^{L}\sqrt{|\blt\hath|}\sqrt{|\clt\hatm|}\left(\sqrt{\|\vxi\|_\infty}\right)\notag\\
	\leq &\tfrac{2}{L}\sqrt{\|\mB\hath\|_1\|\mC\hatm\|_1}\left(\sqrt{\|\vxi\|_\infty}\right)\label{eq:interim-eq10} \\
	\leq &\tfrac{2}{L}\sqrt{L\|\mB\|\|\hath\|_2\|\mC\|\|\hatm\|_2}\sqrt{\|\vxi\|_\infty}\notag\\
	\leq &18\sqrt{\|\hath\|_2\|\|\hatm\|_2}\sqrt{\|\vxi\|_\infty},\label{eq:interim-eq3}
\end{align}
where \eqref{eq:interim-eq10} follows from Cauchy-Schwartz inequality and \eqref{eq:interim-eq3} holds with probability $1-\mathrm{e}^{-cL}$ because, by Corollary 5.35 in \cite{vershynin10in}, $\|\mB\|\leq 3\sqrt{L}$ with probability $1-2\mathrm{e}^{-\tfrac{L}{2}}$. and $\|\mC\| \leq 3\sqrt{L}$ with probability $1-2\mathrm{e}^{-\tfrac{L}{2}}$ as well.} We now lower bound $I$ in \eqref{eq:interim-eq3}. Let $\psi_t(s) := (s)_+-(s-t)_+$. Using the fact that $\psi_t(s) \leq (s)_+$, and that for every $\alpha, t \geq 0$, and $s \in \R$, $\psi_{\alpha t}(s) = t\psi_{\alpha}(\tfrac{s}{t})$, we have
\begin{align}\label{eq:interim-eq2}
I & \geq \tfrac{1}{L}\sum_{\ell=1}^L\psi_{\tau \|(\dh,\dm)\|_2}\left(\langle \vz_\l,(\dh,\dm)\rangle\right) \notag\\
& = \|(\dh,\dm)\|_2\cdot\tfrac{1}{L}\sum_{\ell=1}^L\psi_{\tau}\left(\left\< \vz_\l,\tfrac{(\dh,\dm)}{\|(\dh,\dm)\|_2}\right\>\right)\notag\\
& = \|(\dh,\dm)\|_2\cdot\tfrac{1}{L} \Bigg[\sum_{\ell=1}^L\E\psi_{\tau}\left(\left\< \vz_\l,\tfrac{(\dh,\dm)}{\|(\dh,\dm)\|_2}\right\>\right) - \notag\\
& \quad \sum_{\ell=1}^L\bigg(\E\psi_{\tau}\left(\left\< \vz_\l,\tfrac{(\dh,\dm)}{\|(\dh,\dm)\|_2}\right\>\right)-\psi_{\tau}\left(\left\< \vz_\l,\tfrac{(\dh,\dm)}{\|(\dh,\dm)\|_2}\right\>\right)\bigg)\Bigg].
\end{align}
 The proof mainly relies on lower bounding the right hand side above uniformly over all $(\dh,\dm) \in \setD$. To this end, define a centered random process $\mathcal{R}(\mB,\mC)$ as follows
\begin{align*}
\mathcal{R}(\mB,\mC):=&\sup_{(\dh,\dm)\in \setD}\tfrac{1}{L}\sum_{\ell=1}^L\bigg[\E\psi_{\tau}\bigg(\left\< \vz_\l,\tfrac{(\dh,\dm)}{\|(\dh,\dm)\|_2}\right\>\bigg)\\
&\quad\qquad\qquad\qquad -\psi_{\tau}\bigg(\left\< \vz_\l,\tfrac{(\dh,\dm)}{\|(\dh,\dm)\|_2}\right\>\bigg)\bigg],
\end{align*}
and an application of bounded difference inequality \cite{mcdiarmid1989method} yields that $\mathcal{R}(\mB,\mC) \leq \E \mathcal{R}(\mB,\mC) + t\tau/\sqrt{L}$  with probability at least $1-\mathrm{e}^{-2Lt^2}$. It remains to evaluate $\E  \mathcal{R}(\mB,\mC)$, which after using a simple symmetrization inequality \cite{van1997weak} yields 
\begin{align}
\E \setR(\mB,\mC)\leq & 2\E \sup_{(\dh,\dm)\in \setD}\tfrac{1}{L}\sum_{\ell=1}^L\varepsilon_\ell \psi_{\tau}\bigg(\left\< \vz_\l,\tfrac{(\dh,\dm)}{\|(\dh,\dm)\|_2}\right\>\bigg),
\end{align}
where $\varepsilon_1, \varepsilon_2, \ldots, \varepsilon_L$ are independent Rademacher random variables. Using the fact that $\psi_t(s)$ is a contraction: $|\psi_t(\alpha_1)-\psi_t(\alpha_2)| \leq |\alpha_1-\alpha_2|$ for all $\alpha_1, \alpha_2 \in \R$, we have from the Rademacher contraction inequality \cite{ledoux2013probability} that 
\begin{align}\label{eq:random-process}
\E \sup_{(\dh,\dm)\in \setD}&\tfrac{1}{L}\sum_{\ell=1}^L\varepsilon_\ell \psi_{\tau}\bigg(\left\< \vz_\l,\tfrac{(\dh,\dm)}{\|(\dh,\dm)\|_2}\right\>\bigg) \notag\\
& \leq \E \sup_{(\dh,\dm)\in \setD}\tfrac{1}{L}\sum_{\ell=1}^L\varepsilon_\ell \left\< \vz_\l,\tfrac{(\dh,\dm)}{\|(\dh,\dm)\|_2}\right\>.
\end{align}
In addition, using the facts that $t\mathbf{1}(s\geq t) \leq \psi_t(s)$ it follows 
\begin{align}\label{eq:tail-prob}
\E \psi_{\tau}\left( \left\< \vz_\l,\tfrac{(\dh,\dm)}{\|(\dh,\dm)\|_2}\right\>\right) \geq &\tau\E \left[ \mathbf{1}\left(\left\< \vz_\l,\tfrac{(\dh,\dm)}{\|(\dh,\dm)\|_2}\right\>\geq \tau\right)\right] \notag\\
& = \tau\PP\left(\left\< z_\l,\tfrac{(\dh,\dm)}{\|(\dh,\dm)\|_2}\right\>\geq \tau\right)
\end{align}
Plugging \eqref{eq:tail-prob}, and \eqref{eq:random-process} in \eqref{eq:interim-eq2}, we have 
\begin{align}\label{eq:interim-eq4}
I \geq & \quad \tau\|(\dh,\dm)\|_2\PP\left(\left\< \vz_\l,\tfrac{(\dh,\dm)}{\|(\dh,\dm)\|_2}\right\> \geq \tau\right) \notag\\
& -\|(\dh,\dm)\|_2 \Big(2\E \sup_{(\dh,\dm)\in \mathcal{D}}\tfrac{1}{L}\sum_{\ell=1}^L\varepsilon_\ell\left\< \vz_\l,\tfrac{(\dh,\dm)}{\|(\dh,\dm)\|_2}\right\>+\tfrac{t\tau}{\sqrt{L}}\Big)
\end{align}
Combining this with \eqref{eq:interim-eq3}, we obtain the final result 
\begin{align*}
&\|(\dh,\dm)\|_2\Bigg[\tau\PP\left(\left\< \vz_\l,\tfrac{(\dh,\dm)}{\|(\dh,\dm)\|_2}\right\> \geq \tau\right) \\
& \qquad  - \left(2\E \sup_{(\dh,\dm)\in \mathcal{D}}\tfrac{1}{L}\sum_{\ell=1}^L\varepsilon_\ell\left\< \vz_\l,\tfrac{(\dh,\dm)}{\|(\dh,\dm)\|_2}\right\>+\tfrac{t\tau}{\sqrt{L}}\right)\Bigg]\leq 18\sqrt{\|\hath\|_2\|\|\hatm\|_2}\sqrt{\|\vxi\|_\infty}
\end{align*}
Using the definitions in \eqref{eq:Rademacher-Complexity}, and \eqref{eq:Tail-Prob}, we can write 
\begin{align*}
\|(\dh,\dm)\|_2 \left(\tau \mathfrak{p}_{\tau}(\mathcal{D})- \tfrac{(2\mathfrak{C}(\mathcal{D}) + t\tau)}{\sqrt{L}}\right) \leq 18\sqrt{\|\hath\|_2\|\|\hatm\|_2}\sqrt{\|\vxi\|_\infty}.
\end{align*}

{\change It is clear that  choosing $L \geq \left( \tfrac{2\mathfrak{C}(\mathcal{D})+t\tau}{\tau \mathfrak{p}_\tau(\mathcal{D})}\right)^2$ implies the any feasible descent direction from $(\hath,\hatm)\oplus\mathcal{N}$ is bounded by
\begin{align}
\|(\dh,\dm)\|_2 \leq 18\sqrt{\|\hath\|_2\|\|\hatm\|_2}\sqrt{\|\vxi\|_\infty}
\end{align}
with probability at least $1 - \mathrm{e}^{-c_t L}$. Here $c_t$ is a constant that depends quadratically on $t$. Since $(\dh,\dm) \in\mathcal{N}_{\perp}$, the inequality above only gives us an element, $((1-\beta_0)\hath,(1+\beta_0)\hatm)$ for some $\beta_0 \in \R$, of the set $(\hath,\hatm)\oplus\mathcal{N}$ obeys 
\begin{align}\label{descent bound}
\|(\tildeh,\tildem)-((1-\beta)\hath,(1+\beta)\hatm)\|_2 \leq 18\sqrt{\|\hath\|_2\|\|\hatm\|_2}\sqrt{\|\vxi\|_\infty}.
\end{align}

That is, the solutions $(\tildeh,\tildem)$ cannot waver too far away from the line $(\hath,\hatm)\oplus\mathcal{N}$. We call this norm cylinder constraint as the solution must lie within a cylinder, centered at a line $(\hath,\hatm)\oplus\mathcal{N}$ and of radius given by the r.h.s. of the equation \eqref{descent bound}. Equivalently, a displacement $((1-\beta)\hath,(1+\beta)\hatm)$ of the ground truth $(\hath,\hatm)$ is sufficiently close to $(\tildeh,\tildem)$. Using this fact together with the fact that the feasible hyperbolic set diverges away from the line $(\hath,\hatm)\oplus\mathcal{N}$ for large displacement $\beta$ and touches the line at $\beta =0$, we will conclude in the remaining proof that the Euclidean distance between $(\tildeh,\tildem)$ and the bilinear ambiguity curve corresponding to the ground truth $(\hath,\hatm)$ is bounded. 

We first note that in the case when $\vxi =\vzero$, equation \eqref{descent bound} implies that $(\tildeh,\tildem)$ must be on the line $(\hath,\hatm)\oplus\mathcal{N}_\perp$. Since the only element in the line $(\hath,\hatm)\oplus\mathcal{N}_\perp$ that is feasible is $(\hath,\hatm)$, we conclude that in the noiseless case $(\tildeh,\tildem) = (\hath,\hatm)$.

Now, we use the fact that the noise $\vxi$ is such that $\xi_\l\in[-1,0]$ for every $\l \in [m]$, and there exists an $\l'\in[m]$ such that $\xi_{l'} = 0$. Trivially, the minimizer $(\tildeh,\tildem)$ must lie somewhere in the feasible set specified by the $\l'$ constraint: $\sign(\yl)\vb_{\l'}^\intercal\vh\vc_{\l'}^\intercal\vm\geq|y_{\l'}|$ and $t_{\l'}\vb_{\l'}^\intercal\vh \geq 0$. Define the boundary $\mathcal{B}$ of the feasible set as follows
\begin{align}
	\mathcal{B} :=\left\{(\vh,\vm):\vb_{\l'}^\intercal\vh\vc_{\l'}^\intercal\vm=y_{\l'},\ t_{\l'}\vb_{\l'}^\intercal\vh\geq0 \right\}.
\end{align}
The line $(\hath,\hatm)\oplus\mathcal{N}_\perp$ only touches the feasible set $\mathcal{B}$ at $(\hath,\hatm)$. For a fixed displacement $\beta \in \R$ from $(\hath,\hatm)$, define a segment of the norm cylinder in \eqref{descent bound} as
\begin{equation}
	\mathcal{C}_{\beta} :=\{(\vh,\vm):\|(\vh,\vm)-((1-\beta)\hath,(1+\beta)\hatm)\|_2 \leq 18\sqrt{\|\hath\|_2\|\|\hatm\|_2}\sqrt{\|\vxi\|_\infty}.\}
\end{equation}
Clearly, there exists a $\beta_0$ such that $(\tildeh,\tildem)\in \mathcal{C}_{\beta_0}$. Moreover, there exist a $c>0$ such that $(c\hath,\tfrac{1}{c}\hatm) \in \mathcal{C}_{\beta_0}$. This is because $(\tildeh,\tildem)$ must live in the convex hull of the set $\mathcal{B}$ and the bilinear ambiguity curve corresponding to $(\hath,\hatm)$ is in the set $\mathcal{B}$. Since the distance between any two points in a cross-section of a cylinder is at most twice the radius of the cylinder, we have
\begin{equation}
	\|(\tildeh,\tildem)-(c\hath,\tfrac{1}{c}\hatm)\|_2 \leq 36\sqrt{\|\hath\|_2\|\|\hatm\|_2}\sqrt{\|\vxi\|_\infty}
\end{equation}
for some $c>0$. 
}
\end{proof}

We now compute the Rademacher complexity $\mathfrak{C}(\mathcal{D})$ defined in \eqref{eq:Rademacher-Complexity} of the set of descent directions $\mathcal{D}$ defined in \eqref{eq:setD}.   
 \begin{lemma}\label{lem:Rada}
 	Fix $(\ho,\mo)\in \mathbb{R}^{K \times N}$. Let $\mB \in \mathbb{R}^{L\times K}$ and $\mC \in \mathbb{R}^{L\times N}$ have i.i.d. $\mathcal{N}(0,1)$ entries. Let $\mathcal{D}$ be as defined in \eqref{eq:setD}. Then 
 	\begin{align}\label{eq:complexity-upperbound}
   \mathfrak{C}(\mathcal{D}) &\leq C\sqrt{\big(\|\hatm\|_2^2+\|\hath\|_2^2\big)(S_1+ S_2)\log^2(K+N)}
   \end{align}
     where $C>0$ is an absolute constant.
 \end{lemma}
\begin{proof} We start by evaluating  
	\begin{align}\label{eq:complexity-calc}
	\mathfrak{C}(\mathcal{D}) &= \E \sup_{(\dh,\dm) \in \mathcal{D}} \tfrac{1}{\sqrt{L}} \sum_{\ell=1}^L \varepsilon_\ell \left\<(\clt\hatm\bl,\blt\hath\cl),\tfrac{(\dh,\dm)}{\|(\dh,\dm)\|_2}\right\>\notag\\
	&\leq \E \left\|\tfrac{1}{\sqrt{L}} \sum_{\ell=1}^L \varepsilon_\ell \left(\clt\hatm \bl\vert_{\Gamma_h},\blt \hath \cl\vert_{\Gamma_m}\right)\right\|_2 \cdot \sup_{(\dh,\dm) \in \mathcal{D}}\left\|\tfrac{\left(\dh_{\Gamma_h},\dm_{\Gamma_m}\right)}{\|(\dh,\dm)\|_2}\right\|_2   
	\notag \\
	& \quad + \E  \left\|\tfrac{1}{\sqrt{L}} \sum_{\ell=1}^L \varepsilon_\ell \left(\clt\hatm \bl\vert_{\Gamma^c_h},\blt \hath \cl\vert_{\Gamma^c_m}\right)\right\|_\infty\cdot  \sup_{(\dh,\dm) \in \mathcal{D}}\left\|\tfrac{\left(\dh_{\Gamma_h^c},\dm_{\Gamma_m^c}\right)}{\|(\dh,\dm)\|_2}\right\|_1.
	\end{align}
	First note that on set $\setD$, we have 
	\begin{align*}
	\left\|\tfrac{\big(\dh_{\Gamma_h^c},\dm_{\Gamma_m^c}\big)}{\|(\dh,\dm)\|_2}\right\|_1 \leq \sqrt{S_1+S_2} \left\|\tfrac{\left(\dh_{\Gamma_h},\dm_{\Gamma_m}\right)}{\|(\dh,\dm)\|_2}\right\|_2  \leq \sqrt{S_1+S_2}. 
	\end{align*}
   As for the remaining terms, we begin by writing
	\begin{align*}
	&\E 	\left\|\tfrac{1}{\sqrt{L}} \sum_{\ell=1}^L \varepsilon_{\ell} \left(\clt\hatm \bl\vert_{\Gamma_h},\blt \hath \cl\vert_{\Gamma_m}\right)\right\|_2 \leq \sqrt{\E 	\left\|\tfrac{1}{\sqrt{L}} \sum_{\ell=1}^L \varepsilon_{\ell} \left(\clt\hatm \bl\vert_{\Gamma_h},\blt \hath \cl\vert_{\Gamma_m}\right)\right\|_2^2}\\
	& \qquad\qquad \qquad = \sqrt{\tfrac{1}{L}\sum_{\ell=1}^L \E \left(|\clt\hatm|^2 \|\blt\vert_{\Gamma_h}\|_2^2 +|\bl\hath|^2 \|\cl\vert_{\Gamma_m}\|_2^2\right) }\\ 
	&\qquad\qquad\qquad= \sqrt{\|\hatm\|_2^2 S_1+\|\hath\|_2^2 S_2},
   \end{align*}
   and the second term in \eqref{eq:complexity-calc} is  
   \begin{align*}
   &\E \left\|\tfrac{1}{\sqrt{L}} \sum_{\ell=1}^L \varepsilon_\ell(\blt\hath\cl\vert_{\Gamma_m^c},\clt\hatm\bl\vert_{\Gamma_h^c})\right\|_\infty \leq \sqrt{  \E \left\|\tfrac{1}{\sqrt{L}} \sum_{\ell=1}^L \varepsilon_\ell(\blt\hath\cl\vert_{\Gamma_m^c},\clt\hatm\bl\vert_{\Gamma_h^c})\right\|^2_\infty }\\
   & \qquad\qquad\qquad \leq \sqrt{2e \log (K+N)\cdot \tfrac{1}{L}\sum_{\ell=1}^L\E \max\left\{|\clt\hatm|^2 \|\bl\vert_{\Gamma_h^c}\|_{\infty}^2,|\blt\hath|^2 \|\cl\vert_{\Gamma_m^c}\|_\infty^2\right\}}\\
   & \qquad\qquad\qquad\leq \sqrt{2e\log (K+N)\E\max\{|\vb^{\intercal}\hath|^2\|\vc\vert_{\Gamma_m^c}\|_{\infty}^2, |\vc^{\intercal}\hatm|^2\|\vb\vert_{\Gamma_h^c}\|_{\infty}^2\}}\\
  &\qquad\qquad\qquad\leq C\sqrt{\max\{\|\hath\|_2^2,\|\hatm\|_2^2\}\log^2(K+N)}, 
   \end{align*}
   where the second inequality by the application of Lemma 5.2.2 in \cite{akritas2016topics}, 
and the final equality is due to the fact that $\|\vc\vert_{\Gamma_m^c}\|_{\infty}^2$, and $\|\vb\vert_{\Gamma_h^c}\|_{\infty}^2$  are subexponential and using Lemma 3 in \cite{van2013bernstein}.
Plugging the bounds above back in \eqref{eq:complexity-calc}, we obtain the upper bound on the Rademacher complexity given below
   \begin{align}\label{eq:complexity-upperbound}
   \mathfrak{C}(\mathcal{D}) &\leq C\sqrt{\big(\|\hatm\|_2^2+\|\hath\|_2^2\big)  (S_1+ S_2)\log^2(K+N)}. 
   \end{align}
 \end{proof}
 
Next we compute the tail probability estimate $\mathfrak{p}_{\tau}(\mathcal{D})$ defined in \eqref{eq:Tail-Prob}. 
 \begin{lemma}\label{lem:Tail}
   	Fix $(\ho,\mo)\in \mathbb{R}^{K \times N}$. Let $\mB \in \mathbb{R}^{L\times K}$ and $\mC \in \mathbb{R}^{L\times N}$ have i.i.d. $\mathcal{N}(0,1)$ entries. Let $\mathcal{D}$ be as defined in \eqref{eq:setD} and set $\tau =\sqrt{\tfrac{1}{2}\min\{\|\hath\|_2^2,\|\hatm\|_2^2\}}$ . Then $\mathfrak{p}_{\tau}(\mathcal{D}) \geq \tfrac{1}{8c^4}$ for some absolute constant $c>0$.
  \end{lemma}
\begin{proof}In order to evaulate 
\begin{align}\label{eq:tail-probability}
\mathfrak{p}_{\tau}(\mathcal{D})= \inf_{(\dh,\dm) \in \setD}
   \PP\left(\left\<(\clt\hatm\bl,\blt\hath\cl),\tfrac{(\dh,\dm)}{\|(\dh,\dm)\|_2}\right\>\geq \tau\right)
  \end{align}
   it suffice to estimate the probability $\PP(|\blt\hath\clt\dm+\blt\dh\clt\hatm| \geq \tau)$. Using Paley-Zygmund inequality, we have
   \begin{align*}
   & \PP\left(|\blt\hath\clt\dm+\blt\dh\clt\hatm|^2 \geq \frac{1}{2}\E |\blt\hath\clt\dm+\blt\dh\clt\hatm|^2 \right) \\
   & \qquad\qquad \qquad \geq \frac{1}{4}\cdot \frac{\left(\E |\blt\hath\clt\dm+\blt\dh\clt\hatm|^2\right)^2}{\E |\blt\hath\clt\dm+\blt\dh\clt\hatm|^4}.
   \end{align*}
   Using norm equivalence of Gaussian random variables, we know that $\left(\E |\blt\hath\clt\dm+\blt\dh\clt\hatm|^4\right)^{1/4} \leq c \left(\E |\blt\hath\clt\dm+\blt\dh\clt\hatm|^2\right)^{1/2}$, this implies that 
   \begin{align}
   \PP\left(|\blt\hath\clt\dm+\blt\dh\clt\hatm|^2 \geq\frac{1}{2}\E |\blt\hath\clt\dm+\blt\dh\clt\hatm|^2\right) \geq \frac{1}{4}\cdot \frac{1}{c^4}.
   \end{align}
 Next, we show that $\E |\blt\hath\clt\dm+\blt\dh\clt\hatm|^2 \geq\ \min\{\|\hath\|_2^2,\|\hatm\|_2^2\}\left(\|\dh\|_2^2+\|\dm\|_2^2\right)$. Consider
\begin{align*}
\E |\blt\hath\clt\dm+\blt\dh\clt\hatm|^2 &= \E_{\vb}\E_{\vc} \hath^\top \bl \blt\hath\dm^\top \cl \clt\dm + \dh^\top \bl\blt\dh\hatm^\top\cl \clt\hatm\\ &
\qquad \quad + 2\E_{\vb}\E_{\vc} \dh^\top\bl\blt\hath\dm^\top \cl\clt\hatm \\ & = \E_{\vb} \|\dm\|^2 \hath^\top \bl \blt\hath + \|\hatm\|^2 \dh^\top \bl\blt\dh + 2\dm^\top \hatm \dh^\top\bl\blt\hath\\ 
&=\|\dm\|^2\|\hath\|^2 + \|\hatm\|^2\|\dh\|^2 + 2 \dm^\top \hatm \dh^\top \hath\\
&= \|\dm\|^2\|\hath\|^2 + \|\hatm\|^2\|\dh\|^2 + 2\left(\dm^\top \hatm\right)^2\\
&\geq \min\{\|\hath\|_2^2,\|\hatm\|_2^2\}\left(\|\dh\|_2^2+\|\dm\|_2^2\right)
\end{align*}
where the second and third equalities follow because $\bl$ and $\cl$ contain i.i.d. $\mathcal{N}(0,1)$ entries. The fourth equality follows from the fact $(\dh,\dm) \in \setD \subset \setN_\perp$, and hence $\setD \perp \setN$, which implies that $\dh^\top\hath = \dm^\top\hatm.$ Normalizing by $\|(\dh,\dm)\|_2$, and comparing with \eqref{eq:tail-probability} directly shows that $\tau^2 =\tfrac{1}{2}\min\{\|\hath\|_2^2,\|\hatm\|_2^2\}$, and $\mathfrak{p}_{\tau}(\mathcal{D}) \geq \tfrac{1}{8c^4}$. This completes the proof. 
\end{proof}

We now present a proof of Theorem \ref{thm:Noisy_Main}. In Theorem \ref{thm:Noisy_Main}, the noise satisfy $\xi_\l \geq -1$ which is in contrast to $\xi_\l \in [-1, 0]$ with $\xi_{\ell^{'}} = 0$ for some $\ell^{'} \in \{1,\dots,L\}$ in Lemma \ref{lem:Mendelson}. The key idea is measurements with noise that satisfy $\xi_\l \geq -1$ can be converted to measurements with noise in the interval $[-1, 0]$ with the noise for one of the measurement exactly equal to zero. In order to see this, let 
\begin{align}
	s &= \max_{\l \in [L]} \frac{\yl}{\hat{\yl}} = 1+\max_{\l \in [L]}\xi_{\l}\leq1+\|\vxi\|_\infty,\label{shift}\\
	 \eta_\ell &= \frac{1}{s}(1-s+\xi_\l). \label{eta}
\end{align} 
We then consider the measurements $\yl = s\hat{\yl}(1+\eta_\l)$ for $\l \in [L]$. Because $s\hat{\yl}(1+\eta_\l) = \hat{\yl}(1+\xi_\l)$, the noisy measurements are the same, however the noise may be different.
\begin{proof}[{\bf Proof of Theorem \ref{thm:Noisy_Main}}]
	As the noise of measurements $y_\l = \hat{y_\l}(1 + \xi_\l)$ may not be one-sided as in \eqref{noise-condition}, we consider equivalent measurements $y_\l = s\hat{y_\l}(1 + \eta_\l)$, where $s$ and $\eta_\l$ are as defined in \eqref{shift} and \eqref{eta}, respectively. This turns the $\l_1$-BranchHull program \eqref{eq:BH} into
	\begin{equation}\label{shift BH}
		\begin{aligned}
			\underset{\vh\in \R^K, \vm \in \R^N}{\minimize} \|\vh\|_{1}+\|\vm\|_{1} \hspace{5mm} \text{ s.t. }\hspace{5mm}  &\sl\blt \vh \clt \vm \geq |s\hat{y_{\ell}}(1+\eta_\l)|, \\
			\vspace{-10mm} & t_{\ell}\cdot\blt \vh \geq 0,\ \ell\in \{1,\dots,L\}.
	\end{aligned}	
	\end{equation}	
	First, we note that for all $\l \in \{1,\dots,L\}$,
	\begin{align}
		\eta_\l & =  \frac{1}{s}\left(1+\xi_\l-s\right)\\
		&\leq \frac{1}{s}(s-s)\\
		& = 0,
	\end{align}
	where the first inequality holds because $1+\xi_\l \leq 1+\max_{\ell \in [L]} \xi_\l \leq s$. Second, we have $\eta_\l \geq -1$ for all $\l \in \{1,\dots,L\}$, which follows directly from $\xi_\l \geq -1$ for all $\l$. Third, there exists a $\l'$ such that $\eta_{\l'}=0$. Thus, the noise $\veta$ satisfies \eqref{noise-condition} and by Lemma \ref{lem:Mendelson}, the minimizer $(\tildeh,\tildem)$ of \eqref{shift BH} is unique and if $L\geq \left( \tfrac{2\mathfrak{C}(\mathcal{D})+t\tau}{\tau \mathfrak{p}_\tau(\mathcal{D})}\right)^2$, the minimizer satisfies
	\begin{equation}\label{shift guarantee}
		\left(\left\|\tildeh - c\sqrt{s}\hath\right\|_2^2+\left\|\tildem - c^{-1}\sqrt{s}\hatm\right\|_2^2\right)^{\tfrac{1}{2}} \leq 36\sqrt{s\|\hath\|_2\|\|\hatm\|_2}\sqrt{\|\veta\|_\infty}\end{equation}
	for some $c>0$ with probability at least $1-e^{-c_t L}	$. Furthermore, $c\rightarrow 1$ as $\veta \rightarrow \vzero$. In \eqref{shift guarantee},
	\begin{align}
		s\|\veta\|_\infty & = -\min_{\l \in [L]}\left(1-s+\xi_\l\right)\nonumber\\
			& = s-1-\min_{\l \in [L]}\xi_\l\nonumber\\
			& = \max_{\l \in [L]}\xi_\l-\min_{\l \in [L]}\xi_\l\nonumber\\
			& \leq 2\|\vxi\|_\infty \label{delta bound}
	\end{align}
	 where the first equality holds because $\eta_\l \leq 0$ for all $\l\in\{1,\dots,L\}$. We now compute
	\begin{align}
		& \left(\left\|\tildeh -c\hath\right\|_2^2+\left\|\tildem - c^{-1}\hatm\right\|_2^2\right)^{\frac{1}{2}} \notag\\
		= & \left(\left\|\tildeh - c\sqrt{s}\hath+c\sqrt{s}\hath-c\hath\right\|_2^2+\left\|\tildem - c^{-1}\sqrt{s}\hatm+c^{-1}\sqrt{s}\hatm-c^{-1}\hatm\right\|_2^2\right)^{\frac{1}{2}} \notag\\
		\leq & \left(\left\|\tildeh - c\sqrt{s}\hath\right\|_{2}^2+\left\|\tildem - c^{-1}\sqrt{s}\hatm\right\|_2^2\right)^{\frac{1}{2}}+\left(\left\|c\sqrt{s}\hath-c\hath\right\|_{2}^2+\left\|c^{-1}\sqrt{s}\hatm-c^{-1}\hatm\right\|_2^2\right)^{\frac{1}{2}}\label{triangle 1}\\
		\leq & 36\left(s\|\veta\|_\infty\|\hath\|_2\|\hatm\|_2\right)^{\frac{1}{2}}+ (\sqrt{s}-1)\left(\|c\hath\|_{2}^2+\|c^{-1}\hatm\|_{2}^2\right)^{\frac{1}{2}}\label{apply lemma}\\
		\leq & \left(36\left(\frac{s\|\veta\|_\infty}{2}\right)^{\frac{1}{2}}+ (\sqrt{s}-1)\right)\left(\|c\hath\|_{2}^2+\|c^{-1}\hatm\|_{2}^2\right)^{\frac{1}{2}}\nonumber\\
		\leq & \left(36\left(\|\vxi\|_\infty\right)^{\frac{1}{2}}+ (\sqrt{s}-1)\right)\left(\|c\hath\|_{2}^2+\|c^{-1}\hatm\|_{2}^2\right)^{\frac{1}{2}}\label{eq:interim-eq11}\\
		\leq & \left(36\left(\|\vxi\|_\infty\right)^{\frac{1}{2}}+ (\sqrt{1+\|\vxi\|_\infty}-1)\right)\left(\|c\hath\|_{2}^2+\|c^{-1}\hatm\|_{2}^2\right)^{\frac{1}{2}}\label{eq:interim-eq12}\\
		\leq & 37\sqrt{\|\vxi\|_\infty}\left(\|c\hath\|_{2}^2+\|c^{-1}\hatm\|_{2}^2\right)^{\frac{1}{2}}\label{eq:interim-eq13}
	\end{align}
where \eqref{triangle 1} holds because of triangle inequality, \eqref{apply lemma} holds because of \eqref{shift guarantee}, \eqref{eq:interim-eq11} holds because of \eqref{delta bound}, \eqref{eq:interim-eq12} holds because of \eqref{shift} and \eqref{eq:interim-eq13} holds because for $\sqrt{1+\|\vxi\|_\infty} -1\leq\sqrt{\|\vxi\|_\infty}$. Lastly, we note that $\left( \tfrac{2\mathfrak{C}(\mathcal{D})+t\tau}{\tau \mathfrak{p}_\tau(\mathcal{D})}\right)^2\geq C_\rho\left(\sqrt{S_1+S_2}\log(K+N)+t\right)^2$ because of Lemmas \ref{lem:Rada}, \ref{lem:Tail} and the assumption that $\ho$ and $\mo$ are $\rho$-comparable-effective-sparse as in \eqref{cond:eff sparsity}, for some $\rho \geq 1$. Here, $C_\rho$ is constants that depends quadratically on $\rho$. This completes the proof.
\end{proof}

%


\bibliographystyle{IEEEtran}
\bibliography{Bibliography}

\end{document}